\documentclass{article}

\usepackage{macros}         

\usepackage{verbatim}
\usepackage{soul}
\usepackage{authblk}

\title{Asynchronous decentralized successive convex approximation}
\author[$\dag$]{Ye Tian}
\author[$\dag$]{Ying Sun}
\author[$\dag$]{Gesualdo Scutari}
\affil[$\dag$]{School of Industrial Engineering, Purdue University}

\begin{document}

\maketitle

\begin{abstract}
We study decentralized asynchronous multiagent optimization over networks, modeled as   directed graphs.  The optimization problem consists of minimizing a (nonconvex) smooth function--the sum of the agents' local costs--plus a convex (nonsmooth) regularizer, subject to convex constraints. 		Agents can perform their local computations as well as communicate  with their immediate neighbors at any time, without any form of coordination or centralized scheduling; furthermore,   when solving their local subproblems, they  can use outdated information from their neighbors.
		  We propose the first  distributed algorithm, termed ASY-DSCA, working in such a general asynchronous scenario and applicable to constrained, composite optimization. When the objective function is nonconvex, ASY-DSCA is proved to converge to a stationary solution of the problem at a sublinear rate.
When the problem is convex and satisfies the Luo-Tseng error bound condition, ASY-DSCA converges at an R-linear rate to the optimal solution.  Luo-Tseng (LT) condition is   weaker than   strong convexity of the objective function, and it is satisfied by several nonstrongly convex functions arising from machine learning applications; examples include LASSO and logistic regression problems.    ASY-DSCA is the first distributed algorithm provably achieving linear rate for such a class of problems. \vspace{-0.1cm}
\end{abstract}

\section{Introduction} 
	We consider the following general class of (possibly nonconvex) multiagent \emph{composite} optimization:
	\begin{equation}
		\tag{P}
		\min_{x\in\mathcal{K}} U(x)  \triangleq  \sum_{i\in [I]} f_i(x) +G(x),\vspace{-0.2cm}
		\label{eq:problem}
	\end{equation}
	where $[I] \triangleq \until{I}$ is the set of agents in the system,  $f_i:\mathbb{R}^n\to \mathbb{R}$ is the   cost function of agent $i$, assumed to be smooth but possibly nonconvex; $G:\mathbb{R}^n\to \mathbb{R}$ is  convex  possibly nonsmooth; and $\mathcal{K}\subseteq \mathbb{R}^n$ is a closed convex set.  Each agent has access only to its own objective $f_i$ but not the sum    $ \sum_{i=1}^I f_i$ while $G$ and $\mathcal K$ are common to all the agents.

Problem~\eqref{eq:problem} has found a wide range of applications in machine learning, particularly in supervised learning; examples   include logistic regression, SVM and LASSO, and deep learning.
In these problems,  each $f_i$ is the empirical risk that measures the mismatch  between the model (parameterized by $x$) to be learnt,   and the data set owned \emph{only} by agent $i$.  
 $G$ and $\KK$ plays the role of regularization that restricts the solution space to promote some favorable structure, such as sparsity. 

	Classic distributed learning typically subsumes a master-slave computational architecture wherein the master nodes run  the optimization algorithm gathering the needed information from the workers (cf. Fig.~\ref{fig:comp_centra_decen}-left panel).
	In contrast, in this paper, we consider a decentralized computational architecture,  modeled as a general directed graph that lacks a central controller/master node (see Fig.~\ref{fig:comp_centra_decen}-right panel).
	 Each node can only communicate with its intermediate neighbors. This setting arises naturally when   data are acquired and/or stored     at the node sides.  Examples include resource allocation, swarm robotic control, and multi-agent reinforcement learning \cite{littman1994markov, zhang2018fully}.  Furthermore,  in scenarios where both  architectures are available,  decentralized learning   has the advantage of being robust to single point failures and being communication efficient.  For instance, \cite{lian2017can} compared the performance of  stochastic gradient descent on both architectures; they show that, the two implementations have similar total computational complexity,   while the maximal communication cost per node of the algorithm   running on the decentralized architecture is $\mathcal{O}(\text{degree of network})$, significantly smaller than the $\mathcal{O}(I)$ of the same scheme running on a master-slave system.

\begin{figure}
	\centering{\includegraphics[width=0.6\textwidth]{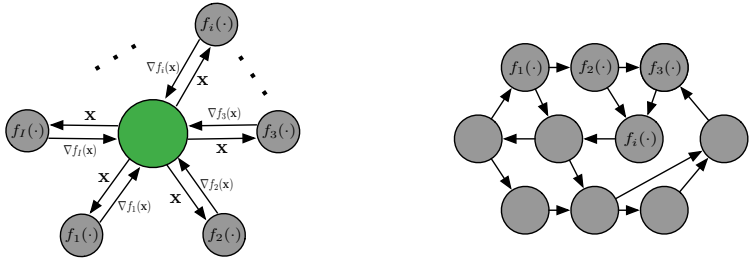}}\vspace{-0.1cm}
	\caption{\small Master-slave   (left panel) vs.  decentralized  (right panel) architectures.}
	\label{fig:comp_centra_decen}\vspace{-0.6cm}
\end{figure}
	As the problem and  network size scale, synchronizing the entire  multiagent system becomes inefficient or infeasible.
	Synchronous schedules require a global clock, which is against the gist of removing the central controller as in  decentralized optimization.
	This calls for the development of {\it asynchronous} decentralized learning algorithms.  In addition, asynchronous modus operandi  brings also  benefits such as   mitigating communication and/or memory-access congestion, saving resources (e.g., energy, computation, bandwidth), and making algorithms more fault-tolerant.  Therefore,  asynchronous decentralized algorithms have the potential to prevail in large scale learning problems. In this paper, we consider the following general decentralized asynchronous setting: \begin{description}\item[(i)]  Agents can perform their local computations as well as communicate (possibly in parallel) with their immediate neighbors at any time, without any form of coordination or centralized scheduling; and
\item[(ii)] when solving their local subproblems, they  can use outdated information from their neighbors, subject to  arbitrary but bounded delays. \end{description} We are not aware of any provably convergence scheme applicable to the envisioned decentralized asynchronous setting and Problem \eqref{eq:problem}--specifically in the presence of constraints or the nonsmooth term $G$--see Sec. \ref{sec:related-works} for a discussion of  related works. This paper fills exactly this gap.
 \vspace{-0.3cm}
\subsection{Main contributions}  \vspace{-0.1cm} Our major  contributions are summarized  next.

	\noindent\textbf{$\bullet$ Algorithmic design:} We introduce \ASYSCA/, the first distributed asynchronous algorithm  [in the sense (i) and (ii) above]   applicable to the \emph{composite, constrained} optimization \eqref{eq:problem}.    \ASYSCA/ builds on successive convex approximation techniques (SCA) \cite{facchinei2015parallel,Scutari_Ying_LectureNote,scutari_PartI,scutari_PartII}--agents solve   strongly convex approximations of \eqref{eq:problem}--coupled with a suitably defined perturbed push-sum mechanism that is robust against asynchrony, whose goal is to track locally and asynchronously the average of agents' gradients. No specific
activation mechanism for the agents' updates, coordination,
or communication protocol is assumed, but only some mild
conditions ensuring that information used in the updates does
not become infinitely old. We remark that SCA offers a unified umbrella to deal efficiently with convex and nonconvex problems \cite{facchinei2015parallel,Scutari_Ying_LectureNote,scutari_PartI,scutari_PartII}: for several problems (P) of practical interest (cf. Sec. \ref{sec:applications}), a proper choice of the  agents' surrogate functions to minimize leads to subproblems that admit a closed form solution (e.g., soft-thresholding and/or projection to the Euclidean ball).  \ASYSCA/ generalizes ASY-SONATA, proposed in the companion paper \cite{Tian_arxiv},  by i) enabling SCA models in the agents' local updates;  and ii) enlarging the class of optimization problems to include constraints and nonsmooth (convex) objectives.

	\noindent\textbf{$\bullet$ Convergence rate:}  Our convergence results are the following:  \textbf{i)} For  general  nonconvex $F$ in \eqref{eq:problem}, a sublinear rate is established for a suitably defined merit function  measuring  both distance of the (average) iterates from stationary solutions and  consensus disagrement; ii) When \eqref{eq:problem} satisfies  the Luo-Tseng (LT) error bound condition~\cite{luo1993error}, we establish   {\it R-linear} convergence   of the sequence generated by  \ASYSCA/ to an optimal solution. Notice that the LT condition is  weaker than strong convexity, which is the common assumption used in the literature to establish linear convergence of distributed (even synchronous) algorithms.   Our interest in the LT condition  is motivated by the fact that
several popular   objective functions  arising from machine learning applications are nonstrongly convex but satisfy the LT error bound; examples include  popular empirical losses in   high-dimensional statistics such as   quadratic  and logistic  losses--see Sec.\ref{sec:eb} for more details.  \ASYSCA/ is the first {\it asynchronous distributed}  algorithm with provably linear rate for such a class of problems over networks;  this result is new even in the synchronous distributed setting.

\noindent\textbf{$\bullet$ New line of analysis: } We put forth  novel  convergence proofs,   whose main novelties are  highlighted next.

\subsubsection{New Lyapunov function for descent}  Our convergence analysis  consists in carefully analyzing the interaction among the consensus, the gradient tracking and the nonconvex-nonsmooth-constrained optimization processes {\it in the asynchronous environment}.   This interaction can be seen as a perturbation that each of these processes induces on
  the dynamics of the others. The challenge is proving that the perturbation generated by one system on the others is of a
 sufficiently small order (with respect to suitably defined metrics), so that convergence can be established and a convergence rate of suitably defined quantities be derived.  Current techniques  from  centralized (nonsmooth) SCA optimization methods \cite{facchinei2015parallel,Scutari_Ying_LectureNote,scutari_PartI,scutari_PartII}, error-bound analysis \cite{luo1993error},  and (asynchronous) consensus algorithms, alone or   brute-forcely put together, do not provide a satisfactory  answer: they   would generate  ``too large'' perturbation errors and do not exploit the interactions among different processes. On the other hand, existing approaches proposed for distributed algorithms are not applicable too (see Sec. \ref{sec:related-works} for a detailed review of the state of the art):  they can neither  deal with asynchrony (e.g., \cite{sun2019convergence}) or be applicable to optimization problems  with a nonsmooth function in the objective and/or constraints.

 To cope with the above challenges our analysis builds on two new Lyapunov functions, one for nonconvex instances of \eqref{eq:problem} and one for convex ones. These functions are carefully crafted to  combine objective  value dynamics  with consensus and gradient errors while accounting for asynchrony and outdated information in the agents' updates. Apart from the specific expression of these functions, a major novelty here is the use in the  Lyapunov functions  of weighting vectors that  {\it endogenously vary  based upon the asynchrony trajectory of the algorithm}--see Sec.\ref{sec:pf_1} (Step 2) and Sec.\ref{sec:pf_2}  (Remark \ref{rmk:Lyapunov})  for   technical details.
  The descent property of the Lyapunov functions is the key step to prove that consensus and tracking errors vanish and further establish the desired converge rate of valid optimality/stationarity measures. 
 \subsubsection{Linear rate under the LT condition}   The proof of linear convergence of \ASYSCA/ under the LT condition is a new contribution of this work. Existing proofs establishing linear rate of distributed synchronous and asynchronous algorithms \cite{shi2015extra,Nedich-geometric,qu2017harnessing,sun2019convergence,alghunaim2019decentralized}   (including   our companion paper \cite{Tian_arxiv}) are not applicable here, as they all leverage  strong convexity of $F$, a property that we do not assume.
 On the other hand,  existing techniques  showing linear rate of {\it centralized} first-order methods under the LT condition  \cite{luo1993error,tseng1991rate}  do not  customize to our {\it distributed}, asynchronous setting. Roughly speaking, this is mainly due to the fact that use of the LT condition in \cite{luo1993error,tseng1991rate}   is subject to proving descent on the objective function along the algorithm iterates, a property that can no longer be guaranteed in the distributed setting, due to the perturbations generated by the consensus and the gradient tracking errors.   Asynchrony complicates further the analysis, as it induces  unbalanced updating frequency of agents and the presence of the outdated information in agents'  local computation. Our proof of linear convergence leverages  the  descent property of the proposed Lyapunov function to be able to invoke the LT condition in our distributed, asynchronous  setting (see  Sec.\ref{sec:pf_1}  for a technical discussion on this matter). \vspace{-0.1cm}
\vspace{-0.3cm}

\subsection{Related works}\label{sec:related-works} \vspace{-0.1cm}\noindent\textbf{On the asynchronous model:} The literature on asynchronous methods is vast; based upon  agents' activation rules and assumptions on   delays,  existing algorithms   can be roughly grouped in    three categories. \textbf{1)} Algorithms  in \cite{wang2015cooperative,li2016distributed,Tsianos:ef,Tsianos:en,lin2016distributed,doan2017impact}  tolerate delayed information but require synchronization among agents, thus fail to meet the asynchronous requirement (i) above.  \textbf{2)} On the other hand, schemes   in \cite{NotarnicolaNotarstefanoTAC17,xu2017convergence,iutzeler2013asynchronous,wei20131,bianchi2016coordinate,hendrikx2019asynchronous,hendrikx2018accelerated} accounts for agents' random (thus uncoordinated) activation; however, upon activation, they must use the most updated information from their neighbors, i.e., no delays are allowed;  hence, they fail to meet requirement (ii).  \textbf{3)}  Asynchronous activations and delays are considered in ~\cite{nedic2011asynchronous,zhao2015asynchronous,kumar2017asynchronous, peng2016arock,wu2016decentralized} and \cite{eisen2017decentralized,bof2017newton,Tian_arxiv,olshevsky2018robust,assran2018asynchronous}, with the former (resp. latter) schemes employing random (resp. deterministic) activations. 
Some restrictions on the form of delays are imposed. Specifically,   \cite{nedic2011asynchronous,zhao2015asynchronous,kumar2017asynchronous,bof2017newton} can only tolerate packet losses (either the information gets lost or is received with no delay);  \cite{assran2018asynchronous}   handles only communication delays (eventually all the transmitted information is received by the intended agent); and \cite{peng2016arock,wu2016decentralized} assume that the agents' activation and delay as independent random variables, which is not  realistic and hard to enforce in practice~\cite{cannelli2016asynchronous}.

The only schemes  we are aware of that are compliant with the asynchronous model    (i) and (ii) are those in \cite{Tian_arxiv,olshevsky2018robust}; however, they are applicable only to  {\it smooth unconstrained} problems.	Furthermore, all  the aforementioned algorithms  but \cite{kumar2017asynchronous,Tian_arxiv} are designed only for {\it convex} objectives $U$.  

\noindent\textbf{On the convergence rate:} Referring to  convergence rate guarantees,  none of the aforementioned  methods is proved to converge linearly in the {\it asynchronous} setting and  when applied to {\it  nonsmooth constrained} problems in the form \eqref{eq:problem}. Furthermore, even restricting the focus to {\it synchronous} distributed methods or   smooth unconstrained instances of (P),  we are not aware of any distributed scheme that provably achieves linear rate without requiring $U$ to be strongly convex; 
we refer to~\cite{sun2019convergence} for a recent literature review of synchronous distributed schemes  belonging to this class.
In the centralized setting, linear  rate can be proved for first order methods under the assumption that  $U$  satisfies some error bound conditions, which are weaker than strongly convexity; see,  e.g.,~\cite{zhang2013linear,bolte2017error,luo1993error,karimi2016linear}. A natural question is whether  such results  can be extended  to (asynchronous) decentralized methods. This paper   provides a positive answer to this open question. \vspace{-0.3cm}

\subsection{Notation}
 The $i$-th vector of the standard basis in $\mathbb{R}^n$ is denoted by $e_i$;  {$x_i$ is the $i$-th entry of a vector $x$; and $A_i$ denotes the $i$-th row of a matrix $A$. Given two matrices (vectors) $A$ and $B$ of same size,  by $A\preccurlyeq B$ we mean  that $B-A$   is a nonnegative matrix (vector).  We do not differentiate between a vector and its transpose when it is the argument of a function/mapping.} The vector of all ones is denoted by  $\mathbf{1}$ (its dimension will be clear from the context).  We use $\norm{\cdot}$ to denote the  Frobenius norm when the argument is   a matrix and the Euclidean norm when applied to a vector;    $\norm{\cdot}_2$ denotes the spectral norm of a matrix. Given $G:\mathbb{R}^n\to \mathbb{R}$,   the proximal mapping  
 is defined as $\text{prox}_{G} (x) \triangleq \argmin_{y\in \mathcal{K}}G(y) + \frac{1}{2} \norm{y-x}_2^2$. Let $\mathcal{K}^*$ denote the set of stationary solutions of \eqref{eq:problem}, and $ \text{dist}(x,\KK^*) \triangleq \min_{y\in \mathcal{K}^*} \|x - y\|$.\vspace{-0.2cm}

	\section{Problem setup}\label{sec:Problem_setup}

	We study Problem~\eqref{eq:problem} under the following  assumptions.
	\begin{assumption}[On Problem~\eqref{eq:problem}] The following hold: 

			(i) The set $\mathcal{K} \subset \mathbb{R}^n$ is nonempty, closed, and convex;

			(ii) Each $f_i:\mathcal{O} \rightarrow \mathbb{R}$  is proper, closed and $l$-smooth, where $\mathcal{O} \supset \mathcal{K}$ is open; $F$ is $L$-smooth  with $L\triangleq I\cdot l$; 

			(iii) $G:\mathcal{K} \to \mathbb{R}$ is convex but possibly nonsmooth;  \text{ and} \quad  (iv) $U$ is lower bounded on $\mathcal{K}$.

		\label{ass:cost_functions}
	\end{assumption}\vspace{-0.1cm}
	Note that each $f_i$ need not  be convex, and  	each agent $i$ knows only its own  $f_i$  but not $\sum_{j\neq i} f_j$. The regularizer $G$ and the  constraint set $\KK$ are common knowledge to all agents.

	To  solve Problem~\eqref{eq:problem},  agents need to leverage message exchanging over the network.
		The communication network of the agents is  modeled as a  fixed, directed
		graph $\GG \triangleq (\mathcal{V},\EE)$.  $\mathcal{V} \triangleq [I]$ is the set of nodes (agents), and $\EE\subseteq \mathcal{V} \times \mathcal{V}$
		is the set of edges ({communication} links). If  $(i,j)\in \EE$, it means that agent $i$
		can send information to agent $j$. We assume that the digraph does not have self-loops.
		We denote by $\nbrs_i^{in}$ the set of \emph{in-neighbors} of agent $i$, i.e.,  
		$\nbrs_i^{in} \triangleq \left\{j \in \mathcal{V} \mid (j,i) \in \EE \right\}$ while   $\nbrs_i^{out} \triangleq \left\{j \in \mathcal{V} \mid (i,j) \in \EE \right\}$
		is the   set of its \emph{out-neighbors}. 
		The following  assumption on the graph connectivity is standard.\vspace{-0.1cm} 
		\begin{assumption}
			The graph $\GG$ is strongly connected.
			\label{ass:strong_conn}\vspace{-0.3cm}
		\end{assumption}
\subsection{Case study: Collaborative supervised learning}\vspace{-0.1cm}\label{sec:applications}
A timely   application of the described decentralized setting and optimization Problem~\eqref{eq:problem} is  collaborative supervised learning. Consider a training data set $\{\left( u_s, y_s\right)\}_{s \in \mathcal{D}} $, where
 $u_s$ is the input feature vector and $y_s$ is the outcome associated to item $s.$  In the envisioned decentralized setting,  data $\mathcal{D}$ are partitioned into $I$ subsets $\{\mathcal{D}_i\}_{i \in [I]}$, each of which belongs to an agent $i\in [I]$.
 The goal is to  learn a mapping $p(\cdot \,; x)$ parameterized by $x \in \mathbb{R}^n$ using   \emph{all} samples in $\mathcal{D}$ by solving  $
 \min_{x \in \mathcal{K}} {1}/{| \mathcal{D}|}\sum_{s \in \mathcal{D}} \ell\left( p(u_s; x), y_s \right) + G(x),
 $ wherein $\ell$ is a loss function  that measures the mismatch between  $ p(u_s; x)$ and $ y_s$; and $G$ and $\KK$ play the role of  regularizing the solution. This problem is an instance of~\eqref{eq:problem} with $f_i(x) \triangleq  {1}/{| \mathcal{D}|} \sum_{s \in \mathcal{D}_i} \ell\left( p(u_s; x), y_s \right)$.
Specific examples of loss functions and regularizers are give next.
 \begin{enumerate}\item[1)] \textbf{ Elastic net regularization for log linear models:}
  $\ell\left( p(u_s; x), y_s \right) \triangleq \Phi (u_s^\top x) - y_s \cdot (u_s^\top x)$ with $\Phi$ convex, $u_s \in \mathbb{R}^n$ and $y_s \in \mathbb{R}$; $G(x) \triangleq  \lambda_1 \norm{x}_1 + \lambda_2 \norm{x}_2^2$ is the elastic net regularizer, which reduces to the LASSO regularizer when $(\lambda_1, \lambda_2) = (\lambda,0)$ or the ridge regression regularizer when $(\lambda_1, \lambda_2) = (0,\lambda)$;

  \item[2)] \textbf{Sparse group LASSO \cite{friedman2010note}:}  The loss function is the same as that in example 1), with $\Phi(t) = t^2/2$; $G(x) = \sum_{S\in \mathcal{J}} w_S \norm{x_S}_2 + \lambda \norm{x}_1$, where $\mathcal{J}$ is a partition of  $[n]$;

\item[3)] \textbf{ Logistic regression:}  $\ell\left( p(u_s; x), y_s \right) \triangleq  \text{ln} ( 1+e^{-y_s \cdot u_s^\top x})$; popular choices of $G(x)$ are $G(x) \triangleq \lambda \norm{x}_1$ or $G(x) \triangleq \lambda \norm{x}_2^2$. The constraint set $\mathcal{K}$ is generally assumed to be bounded. 
\end{enumerate}

For large scale data sets,  solving such learning problems is computationally challenging even if $F$ is convex.  When the problem dimension  $n$ is  larger than the sample size $ |\mathcal{D}|$,
 the Hessian of the  empirical risk   loss $F$ is typically rank deficient and hence  $F$ is not strongly convex. Since linear convergence rate for decentralized methods is established in the literature only under strong convexity, it is  unclear whether such a  fast rate can be achieved under less restrictive conditions,  e.g.,  embracing popular high-dimensional learning problems as those mentioned above. We show next that a positive  answer to this question can be obtained leveraging  the renowned LT error bound, a condition that has been wide explored in the  literature of centralized optimization methods. \vspace{-0.3cm}



\subsection{The Luo-Tseng error bound}\label{sec:eb}

\begin{assumption}(Error-bound conditions \cite{Dembo1984local,pang1986inexact,luo1993error}):
\begin{description}
	\item [(i)] $F$ is convex;

		\item[(ii)] For any $\eta \!> \!\inf_{x \in \mathcal{K}} \!U(x)$, there exists $\epsilon,\kappa > 0$ such that:\vspace{-0.1cm}
		 \begin{align}\label{eq:LT_condition} \hspace{-0.5cm} U(x) \leq \eta \quad \text{and}\quad \norm{x - \text{prox}_G(x - \nabla F(x))} \leq \epsilon  \vspace{-0.5cm}\end{align}\vspace{-0.5cm}
		$$\Downarrow\vspace{-0.4cm}$$
		\begin{align}\label{eq:LT_effect}   \text{dist}(x , \mathcal{K}^*) \leq \kappa \norm{x - \text{prox}_G(x - \nabla F(x))}.\end{align}\end{description}
	\label{ass:cost_functions_eb}\vspace{-0.2cm}
\end{assumption}
 Assumption~\ref{ass:cost_functions_eb}(ii) is a local growth condition on $U$   around $\KK^*$, crucial to  prove  linear rate.  Note that for convex $F$, condition \ref{ass:cost_functions_eb}(ii) is equivalent to other renowned error bound conditions, such as the Polyak-\L{}ojasiewicz \cite{polyak1964gradient,lojasiewicz1963topological}, the quadratic growth~\cite{drusvyatskiy2018error}, and the  Kurdyka-\L{}ojasiewicz~\cite{bolte2017error} conditions.
A broad class of functions satisfying Assumption~\ref{ass:cost_functions_eb} is in the form   $U(x) = F(x)+G(x)$, with $F$ and $G$ such that   (cf. \cite[Theorem~4]{tseng2009coordinate},  \cite[Theorem~1]{zhang2013linear}): 
	\begin{description}\item[(i)] $F(x)= h(A x)$ is $L$-smooth, where $h$ is strongly convex and $A$ is any linear operator;

	\item[(ii)] $G$ is either a polyhedral convex function (i.e., its epigraph is a polyhedral set) or has a specific separable form as  $G(x) = \sum_{S\in \mathcal{J}} w_S \norm{x_S}_2 + \lambda \norm{x}_1$, where $\mathcal{J}$ is a partition of the set $[n]$, and $\lambda$ and $w_S$'s are nonnegative weights (we used $x_S$ to denote the vector whose component $i$ is $x_i$ if $i\in S$, and $0$ otherwise);

	\item[(iii)] $U(x)$ is coercive.\end{description}

It follows that  all  examples listed in  Section~\ref{sec:applications} satisfy Assumption~\ref{ass:cost_functions_eb}. Hence,  the proposed  decentralized asynchronous algorithm, to be introduced, will provably achieve linear rate for such a general classes of problems.  \vspace{-0.3cm}

	\section{Algorithmic development}\label{sec:main} 
	Solving Problem~\eqref{eq:problem} over $\mathcal{G}$   poses  the following challenges: i) $U$ is nonconvex/nonsmooth; ii) each agent $i$ only knows its local loss  $f_i$ but not the global $F$; and iii) agents  perform updates in an asynchronous fashion. Furthermore, it is well established that, when $f_i$ are nonconvex (or convex only in some variable), using convex surrogates for $f_i$ in the agents' subproblems rather than just linearization (as in gradient algorithms) provides more flexibility in the algorithmic design and can enhance practical  convergence \cite{facchinei2015parallel,Scutari_Ying_LectureNote,scutari_PartI,scutari_PartII}. This motivated us to equip  our distributed asynchronous design with   SCA models.

	To address these challenges,
	we develop our algorithm building on   SONATA~\cite{YingMAPR,sun2019convergence}, as   to our knowledge it is the only  synchronous decentralized algorithm for \eqref{eq:problem}   capable to handle challenges i) and ii) and incorporating SCA techniques.   Moreover, when employing a constant step size, it converges linearly to the optimal solution of (P) when $F$  is strongly convex; and sublinearly to the set of stationary points of ~\eqref{eq:problem}, when $F$ is nonconvex.
	We begin briefly reviewing SONATA.\vspace{-0.3cm}


	%
	%
	%

	\subsection{Preliminaries: the  SONATA algorithm \cite{YingMAPR,sun2019convergence}}\label{sec:SONATA}
	Each agent $i$ maintains a local estimate $x_i$ of the common optimization vector $x$, to be updated at each iteration; the $k$-th iterate is denoted by $x_i^k$.
	 The specific procedure put forth by SONATA is given in Algorithm~\ref{alg:SONATA} and briefly described next.

	\begin{algorithm}[!ht]
		\caption{The SONATA Algorithm}
		\noindent \textbf{Data:} For all agent $i$ and $\forall j \in \mathcal{N}_i^{in}$, $x_i^0 \in \mathbb{R}^n$, ${z}_i^0 = y_i^0 = \nabla f_i(x_i^0)$, $\phi_i^0 = 1$. Set $k = 0$.
		\begin{algorithmic}
			\While{a termination criterion is not met, \emph{each agent $i \in [I]$}}
			\State \hspace{-0.3cm}(S.1) Local optimization:
			\begin{subequations}
				\begin{align}
					\begin{split}
			\label{eq:loc_opt}
 & \widetilde{x}_{i}^k     =   \underset{x\in \mathcal{K}}{\mathrm{argmin}}~\big\{\widehat{U}_{i} \big(x;x_{i}^k, I \,y_{i}^k - \nabla f_{i}(x_{i}^{k}) \big) \triangleq \\
 &  \quad  \widetilde{f}_i(x;x_i^k) +( I \,y_{i}^k - \nabla f_{i}(x_{i}^{k}))^\top \left(x - x_{i}^{k}\right) + G(x) \big\},
					\end{split}\\
					\begin{split}
						\label{eq:relaxation}	&  v_{i}^{k+1}  =   x_{i}^k + \gamma \left( \widetilde{x}_{i}^k - x_{i}^k\right).
					\end{split}
				\end{align}
			\end{subequations}
			\State \hspace{-0.3cm}(S.2) Consensus step:
			\vspace{-1ex}
			\begin{align}\label{eq:sonata_consensus}
				x_{i}^{k+1} = w_{ii}v_{i}^{k+1} + \sum_{j\in \mathcal{N}_{i}^{{in}}} w_{ij} v_{j}^{k+1}.
			\end{align}
			\State \hspace{-0.3cm}(S.3) Gradient tracking: \vspace{-1ex}
			\begin{align}\label{eq:sonata_tracking}
				\begin{split}
				&	z_i^{k+1}  = \sum_{j = 1}^I a_{ij} \left( z_j^k + \nabla f_j (x_j^{k+1}) - \nabla f_j(x_j^k) \right) ,\\
				&	\phi_i^{k+1}  =  \sum_{j = 1}^I a_{ij} \phi_j^k, \quad y_i^{k+1} = \frac{z_i^{k+1}}{\phi_i^{k+1}}.
				\end{split}
			\end{align}
			\State   $k\leftarrow k+1$
			\EndWhile
		\end{algorithmic}\label{alg:SONATA}
	\end{algorithm}

	\textbf{(S.1): Local optimization.}
	At each iteration $k$, every agent $i$ locally solves a strongly convex approximation of Problem~\eqref{eq:problem} at $x_{i}^k$, as given in~\eqref{eq:loc_opt},
	where $\widetilde{f}_i: \KK \times \KK \to \mathbb{R}$ is a so-called SCA surrogate of $f_i$, that is, satisfies  Assumption~\ref{assump:surrogate} below.   	The second term in~\eqref{eq:loc_opt}, ${(I y_{i}^k - \nabla f_i (x_i^k))}^\top \left(x - x_{i}^{k}\right)$, serves as a first order approximation of $\sum_{ j \neq i} f_j(x)$
	unknown to agent $i$, wherein
	$I y_{i}^k$  tracks the sum gradient $\sum_{j =1}^I \nabla f_j (x_i^k)$ (see step (S.3)). 	We then employ a relaxation step~\eqref{eq:relaxation} with step size $\gamma$.\vspace{-0.1cm}
	\begin{assumption}\label{assump:surrogate}
		 $\widetilde{f}_i: \KK \times \KK \to \mathbb{R}$ satisfies:\begin{description}
		 	\item[(i)] $\nabla \widetilde{f}_i(x;x) = \nabla f_i(x)$ for all $x\in\mathcal{K}$;

			\item[(ii)] $\widetilde{f}_i(\cdot; y)$ is uniformly strongly convex on $\mathcal{K}$ with constant $\widetilde{\mu}>0$;

			\item[(iii)] $\nabla \widetilde{f}_i(x;\cdot)$ is uniformly Lipschitz continuous on $\mathcal{K} $ with constant $\widetilde{l}$.
		 \end{description}
	\end{assumption}
The choice of $\widetilde{f}_i$ is quite flexible. For example, one can construct a proximal gradient type update~\eqref{eq:loc_opt} by linearizing $f_i$ and adding  a proximal term;   if $f_i$ is a DC function, $\widetilde{f}_i$ can retain the convex part of $f_i$ while linearizing the nonconvex part. We refer to~\cite{facchinei2015parallel,Scutari_Ying_LectureNote,scutari_PartI,scutari_PartII} for more details on the choices of $\widetilde{f}_i$, and Sec. \ref{sec:simu} for specific examples used in our experiments.



	\textbf{(S.2): Consensus.} This steps aims at enforcing consensus   on the local variables $x_i$ via    gossiping. Specifically, after the local optimization step, each agent $i$ performs a consensus update~\eqref{eq:sonata_consensus} with mixing matrix $W = (w_{ij})_{i,j=1}^I$ satisfying the following assumption.
	\begin{assumption}\label{assumption_weights} The weight matrices ${W}\triangleq (w_{ij})_{i,j=1}^I$ and ${A}\triangleq (a_{ij})_{i,j=1}^I$ satisfy (we will write ${M}\triangleq (m_{ij})_{i,j=1}^I$ to denote either ${A}$ or ${W}$ and $\mathbf{1} \in \mathbb{R}^I$ is a vector of all ones):\begin{description}
			\item[(i)] $\exists \, \bar{m}>0$ such that: $m_{ii} \geq \bar{m}$,  $\forall i \in \mathcal{V}$; $m_{ij} \geq \bar{m}$,  for all $(j,i) \in \mathcal{E}$;  and $m_{ij}=0$, otherwise;

		\item[(ii)] ${W}$ is row-stochastic, that is,  ${W}\,\mathbf{1}=\mathbf{1}$;  \text{ and }  iii) ${A}$ is column-stochastic, that is,  ${A}^\top\,\mathbf{1}=\mathbf{1}$.\end{description}
	\end{assumption}
	 Several choices for $W$ and $A$  are available; see, e.g.,  ~\cite{sayed2014adaptation}.
 Note that SONATA  uses a row-stochastic matrix $W$ for the consensus update and a  column-stochastic matrix $A$ for the gradient tracking. In fact, for general  digraph, a doubly stochastic matrix compliant with the graph might not exist 
 while one can always build compliant row or column stochastic matrices. These weights can be determined locally by the agents, e.g., once its in- and out-degree can be estimated.

	\textbf{(S.3): Gradient tracking.} This step updates $y_i$ by employing a perturbed push-sum algorithm with weight matrix $A$ satisfying Assumption~\ref{assumption_weights}.   This step aims to track the average
	gradient $(1/I) \sum_{i=1}^I \nabla f_i(x_i)$ via $y_i$.
		In fact, using the column stochasticity of $A$ and applying the telescopic cancellation, one can  check that the  following   holds:\vspace{-0.1cm}
\begin{equation}\vspace{-0.1cm}
 \label{eq:sum_preservation}
\sum_{i = 1}^I \phi_i^k  =\sum_{i = 1}^I \phi_i^0 = I, \quad  \sum_{i = 1}^I z_i^k   = \sum_{i=1}^{I} \nabla f_i(x_i^k). \vspace{-0.1cm}
 \end{equation}
 It can be shown that for all $i \in [I]$, $z_i^k$ and $\phi_i^k$ converges to $\xi_i^k \cdot \sum_{i=1}^{I} z_i^k$ and $\xi_i^k \cdot  \sum_{i = 1}^I \phi_i^k$, respectively, for some $\xi_i^k >0$~\cite{nedic2014distributed}. Hence,   $y_i^k = z_i^k / \phi_i^k$ converges to $(1/I) \sum_{i=1}^I \nabla f_i(x_i^k)$, employing the desired gradient tracking.

 Notice that the extension of the gradient tracking to the asynchronous setting is not trivial, as the ratio consensus property discussed above no longer holds if
agents naively  perform their updates  using in ~\eqref{eq:sonata_tracking} delayed information.
 In fact, 
 packets sent by an agent, corresponding to the summand in~\eqref{eq:sonata_tracking}, may get lost.
 This breaks the equalities in~\eqref{eq:sum_preservation}. Consequently, the ratio $y_i^k$ cannot correctly track the average gradient. 
 To cope with this issue, our approach is to replace step~(S.3)  by the asynchronous gradient tracking mechanism developed in~\cite{Tian_arxiv}.\vspace{-0.1cm} 
 \vspace{-0.2cm}

	\subsection{Asynchronous decentralized SCA  (\ASYSCA/)}\label{eq:asy_sca}
We now break the synchronism in SONATA and propose \ASYSCA/ (cf.   Algorithm~\ref{alg:ASYSCA}). All agents update asynchronously and continuously without coordination, possibly using delayed information from their neighbors.   
	\begin{algorithm}[!ht]
		\caption{The \ASYSCA/ Algorithm}
		\noindent \textbf{Data:} For all agent $i$ and $\forall j \in \mathcal{N}_i^{\text{in}}$, $x_i^0 \in \mathbb{R}^n$, ${z}_i^0 =y_i^0 = \nabla f_i(x_i^0)$, $\phi_i^0 = 1$,  {$\tilde{{\rho}}_{ij}^0 = 0$},  $ \tilde{\sigma}_{ij}^0 = 0$, $\tau_{i j}^{-1} = -D$.  And for $t = -D, -D+1, \ldots, 0$, ${\rho}_{ij}^t = 0$, $\sigma_{ij}^t = 0$, $v_i^t = 0$.  Set $k = 0$.
		\begin{algorithmic}
			\While{a termination criterion is not met}\smallskip

			\State  Pick: $\quad (i^k, {d}^k)$;\smallskip

			\State Set:
			$ \quad
			\tau_{i^k j}^k = \max(\tau_{i^k j}^{k-1}, k-d_j^k), \quad \forall j \in \mathcal{N}_{i^k}^{\text{in}};
			$\smallskip

			\State (S.1) Local optimization:
			\begin{align}\label{eq:sca}
				\begin{split}
					& \widetilde{x}_{i^k}^k     =    \underset{x\in \mathcal{K}}{\mathrm{argmin}} \quad \widehat{U}_{i^k} \left(\,x;\,x_{i^k}^k, \, I \,y_{i^k}^k - \nabla f_{i^k}(x_{i^k}^{k}) \right), \\
					&  v_{i^k}^{k+1}  =   x_{i^k}^k + \gamma \left( \widetilde{x}_{i^k}^k - x_{i^k}^k\right);
				\end{split}
			\end{align}
			\State (S.2) Consensus step (using delayed information): \\
			\begin{equation}  \label{eq:asy-consensus}
			\qquad x_{i^k}^{k+1} = w_{i^ki^k}v_{i^k}^{k+1} + \sum_{j\in \mathcal{N}_{i^k}^{\text{in}}} w_{{i^k}j} v_{j}^{\tau_{i^k j}^k};
			\end{equation}
			\State (S.3) Robust gradient tracking:
			\smallskip\begin{equation}\label{eq:asy-track}
			\begin{aligned}
				\qquad y_{i^k}^{k+1}&= \mathcal{F} \Big(i^k, \, k, \, ( \rho_{i^k j}^{\tau_{i^k j}^k} )_{j\in \mathcal{N}_{i^k}^{\text{in}}}, \, ( \sigma_{i^k j}^{\tau_{i^k j}^k} )_{j\in \mathcal{N}_{i^k}^{\text{in}}},\\
				&\quad  \qquad \nabla f_{i^k}(x_{i^k}^{k+1}) - \nabla f_{i^k}(x_{i^k}^{k}) \Big)
			\end{aligned}\end{equation}
			\State Untouched state variables shift to state $k+1$ while keeping the same value;    $k\leftarrow k+1.$
			\EndWhile
			\item[]

			\Procedure{\FuncName/}{$i, k, ( \rho_{i j} )_{j\in \mathcal{N}_{i}^{\text{in}}},( \sigma_{i j} )_{j\in \mathcal{N}_{i}^{\text{in}}}, {\epsilon} $}
			\State Sum step:\vspace{-0.1cm}
		\begin{equation}\label{eq:asy-sum-step}	\begin{aligned}
			& z_{i}^{k+\frac{1}{2}}  = z_{i}^{k}  +  \sum_{j \in \mathcal{N}_{i}^{\text{in}}}  \left(\rho_{{i}j}-\tilde\rho_{ij}^k\right) + {\epsilon},  \\
			& \phi  _{i}^{k+\frac{1}{2}}  = \phi_{i}^{k}  + \sum_{j \in \mathcal{N}_{i}^{\text{in}}}  \left({\sigma}_{{i}j} - \tilde{\sigma}_{ij}^k \right);
			\end{aligned}\end{equation} \vspace{-0.4cm}
			\State Push step:
		\begin{equation}\label{eq:asy-push-step}
			\hspace{-0.3cm}\begin{aligned}
				& z_{i}^{k+1}    =  a_{ii}\, z_{i}^{k+\frac{1}{2}}  , \quad \phi_{i}^{k+1}    =  a_{ii}\, \phi_{i}^{k+\frac{1}{2}}; \quad \forall j\in \mathcal{N}_{i}^{\text{out}},  \\
				& \rho_{j {i}}^{k+1}   = \rho_{j {i}}^{k} + a_{j {i}} \,z_{i}^{k+\frac{1}{2}},  \quad  {\sigma}_{j {i}}^{k+1}   = {\sigma}_{j {i}}^{k} + a_{j {i}} \,\phi_{i}^{k+\frac{1}{2}};  \end{aligned}	\end{equation}
			\State Mass-Buffer update:
			\begin{equation}
			 \tilde\rho_{{i}j}^{k+1} = \rho_{{i}j},  \quad   \tilde{\sigma}_{{i}j}^{k+1} = {\sigma}_{{i}j},   { \quad \,\ \forall j \in \mathcal{N}_{i}^{\text{in}}};\label{eq:mass-buff-update}
			\end{equation}
			\State \Return  $\quad    z_{i}^{k+1} / \phi_{i}^{k+1}.$

			\EndProcedure
		\end{algorithmic}\label{alg:ASYSCA}
	\end{algorithm}
More specifically, 	a global iteration counter $k$, unknown to the agents, is introduced, which increases by $1$ whenever a variable of the multiagent system changes.
	Let $i^k$ be the agent triggering iteration $k \to k+1$; it executes Steps (S1)-(S.3) (no necessarily withih the same activation), as described below. 

	\textbf{(S.1): Local optimization.} Agent $i^k$ solves  the strongly convex optimization problem (\ref{eq:sca}) based on the local surrogate    $\widehat{U}_{i^k}$.
 {It is tacitly assumed that  $\widehat{U}_{i^k}$ is chosen so that (\ref{eq:sca})  is  simple to solve (i.e.,  the  solution can be computed in closed form or efficiently). Given the solution  $\widetilde{x}_{i^k}^k$, $v_{i^k}^{k+1}$ is generated. }

	\textbf{(S.2): Consensus.}
	Agent $i^k$ may receive delayed variables from its in-neighbors $j \in \mathcal{N}_{i^k}^{in}$, whose iteration index is  $k - d_j^k$. 
	To perform its update, agent $i^k$  first  sorts the ``age'' of  all the received variables from agent $j$ since $k=0$, and then picks the most recently generated one. This is implemented  maintaining a local counter $\tau_{i^k j}$, updated  recursively as $\tau_{i^k j}^k = \max(\tau_{i^k j}^{k-1}, k-d_j^k)$. Thus,  the  variable agent $i^k$ uses from $j$ has iteration index $\tau_{i^k j}^k$.
	Since the  consensus algorithm is robust against asynchrony~\cite{Tian_arxiv}, we simply adopt the  update of SONATA [cf.~\eqref{eq:sonata_consensus}] and replace $v_j^k$ by its delayed version $v_j^{\tau_{i^k j}^k}$.

	\textbf{(S.3): Robust gradient tracking.}
	As anticipated in  Sec.~\ref{sec:SONATA}, the packet loss caused by asynchrony breaks the  sum preservation property \eqref{eq:sum_preservation} in SONATA. If treated in the same way as the $x$ variable in \eqref{eq:asy-consensus}, $y_i$ would fail to track  $(1/I) \sum_{i = 1}^{I} \nabla f_i (x_i)$. To cope with this issue, we leverage the  asynchronous sum-push scheme (P-ASY-SUM-PUSH) introduced in our companion paper \cite{Tian_arxiv} and update the $y$-variable as in \eqref{eq:asy-track}.
	Each agent $i$ maintains  mass counters $(\rho_{ji}, \sigma_{ji})$ associated  to $(z_i, \phi_i)$  that record the cumulative mass generated by $i$  for $j \in \mathcal{N}_i^{out}$ since $k=0$; and
	  transmits $(\rho_{ji}, \sigma_{ji})$. 
	In addition, agent $i$ also maintains  buffer variables $(\tilde{{\rho}}_{ij}, \tilde{\sigma}_{ij})$
	to track the latest mass counter $(\rho_{ij}, \sigma_{ij})$  from $j \in \mathcal{N}_i^{in}$ that has been used in its update.
	We  describe now the update of $z$ and $\rho$;    $\phi$ and $\sigma$ follows similar steps. 
	 For notation simplicity, let $i= i^k$ update. It first performs the sum step \eqref{eq:asy-sum-step} using  a possibly delayed mass counter $\rho_{i j}^{\tau_{i j}^k}$ received from $j$. By computing the difference  $\rho_{i j}^{\tau_{i j}^k} - \tilde{\rho}_{ij}^k$, it collects the sum of the $a_{ij} z_j$'s generated by $j$ that it has not yet added. Agent $i$ then sums them together with a gradient correction term (perturbation) $\epsilon = \nabla f_i (x_i^{k+1}) - \nabla f_i (x_i^{k}) $  to its current state variable $z_i^k$  to form the intermediate mass $z_i^{k+\frac{1}{2}}$. Next, in the push step \eqref{eq:asy-push-step}, agent $i$ splits  $z_i^{k+\frac{1}{2}}$, maintaining $a_{ii} z_i^{k+\frac{1}{2}}$ for itself and accumulating  $a_{ji} z_i^{k+\frac{1}{2}}$ to its local mass counter $\rho_{ji}^k$, to be transmit to $j \in \mathcal{N}_i^{out}$.
Since the last mass counter agent $i$   processed is $\rho_{i j}^{\tau_{i j}^k}$, it sets $\tilde{\rho}_{ij} =\rho_{i j}^{\tau_{i j}^k} $ [cf. \eqref{eq:mass-buff-update}].  Finally, it outputs $y_i^{k+1} \!=\! {z_i^{k+1}}/{\phi_{i}^{k+1}}$.
\vspace{-0.1cm}

	\section{Convergence of \ASYSCA/}\label{sec:con_analy}

	We study \ASYSCA/  under the  asynchronous model below.\vspace{-0.1cm}
	\begin{assumption}[Asynchronous model] \label{ass:delays}Suppose:		\begin{description}
	\item[(i)]  $\exists$  $0<T<\infty$ such that  {$\cup_{t=k}^{k+T-1} i^t=\mathcal{V}$}, for all $k\in \mathbb{N}_+$; 	 \item[(ii)]  $\exists$ $0<D<\infty$ such that $0 \leq d_j^k\leq D$, for all $j\in \mathcal{N}_{i^k}^{in}$ and $k\in \mathbb{N}_+$.\vspace{-0.1cm}\end{description}
	\end{assumption}

	Assumption~\ref{ass:delays}(i) is an essentially cyclic rule stating that  within $T$ iterations all agents will have updated at least once, which guarantees that all of them participate ``sufficiently often''. Assumption~\ref{ass:delays}(ii) requires bounded delay--old information must eventually be purged by the system. This asynchronous model is general and imposes no coordination among agents or specific communication/activation protocol--an extensive discussion on specific implementations and communication protocols satisfying Assumption~\ref{ass:delays} can be found for ASY-SONATA in the companion paper \cite{Tian_arxiv} and  apply also to  \ASYSCA/; we thus omit here further details.

The convergence of  \ASYSCA/  is established under two settings, namely: i) convex $F$   and error bound Assumption~\ref{ass:cost_functions_eb} (cf. Theorem \ref{thm:linear_const}); and ii) general nonconvex $F$ (cf. Theorem \ref{thm:sublinear_const}).

	\begin{theorem}[Linear convergence]\label{thm:linear_const}
		Consider~(P) under Assumption ~\ref{ass:cost_functions} and~\ref{ass:cost_functions_eb}, and let $U^\star$ denote the optimal function value. Let  {$\{(x_i^k)_{i = 1}^ I\}_{k \in \mathbb{N}}$} be the sequence generated by Algorithm~\ref{alg:ASYSCA}, under Assumption~\ref{ass:strong_conn}, \ref{ass:delays}, and with weight matrices   ${W}$ and ${A}$   satisfying Assumption~\ref{assumption_weights}.  Then, there exist a constant $\bar{\gamma}_{cvx} >0$ and a solution $x^\star$ of~\eqref{eq:problem} such that if   $\gamma \leq \bar{\gamma}_{cvx}$, it holds 		\[
		\| U(x_i^k) - U(x^\star) \|  = \mathcal{O}(\lambda^k), \quad \| x_i^k - x^\star \|  = \mathcal{O}\left((\sqrt{\lambda})^k \right),
	\] for all $i \in \mathcal{V}$ and  some $\lambda \in (0,1)$.
\qed
	\end{theorem}

 Theorem~\ref{thm:linear_const} establishes the first linear convergence result of a distributed (synchronous or asynchronous)  algorithm over networks  without requiring strong convexity but the   weaker LT condition. 
Linear convergence is achieve on both function values and sequence iterates.

We consider now the nonconvex setting. To  measure the progress of \ASYSCA/ towards stationarity, we introduce the merit function \begin{equation} \label{eq:merit} \begin{aligned}  M_F(x^k) \triangleq  \max \big\{ \| \bar{x}^k - \prox_G (\bar{x}^k - \nabla F (\bar{x}^k))\|^2,   \sum_{i=1}^I \| x_i^k - \bar{x}^k\|^2 \big\}, \end{aligned} \end{equation} where  $\bar{x}^k \triangleq (1/I) \cdot \sum_{i=1}^I x_i^k$, and $\prox_G$ is the prox operator (cf. Sec.~\ref{sec:eb}). $M_F$ is a valid merit function since it is continuous and $M_F(x^k) = 0$ if and only if all the $x_i$'s are consensual and stationary. The following theorem shows that  $M_F(x^k)$ vanishes at sublinear rate.

	\begin{theorem}[Sublinear convergence]\label{thm:sublinear_const}
		Consider~(P) under Assumption~\ref{ass:cost_functions} (thus possibly nonconvex). Let   {$\{(x_i^k)_{i = 1}^I \}_{k \in \mathbb{N}_0}$} be the sequence generated by Algorithm~\ref{alg:ASYSCA}, in the same setting of Theorem~\ref{thm:linear_const}. Given $\delta>0$,  let $T_{\delta}$ be the first iteration $k\in \mathbb{N}$ such that $M_F (x^k)  \leq \delta$.
		Then, there exists a  $\bar{\gamma}_{ncvx} >0$, such that if $ \gamma \leq \bar{\gamma}_{ncvx}$,   $T_\delta = \mathcal{O} (1/\delta)$. \qed
	\end{theorem}

  The expression of the step-size can be found in \eqref{eq:noncon_step}. 
  \vspace{-0.3cm}

	\section{Numerical Results}\label{sec:simu}
We test  \ASYSCA/ on a LASSO problem (a convex instance of \eqref{eq:problem}) and an M-estimation problem (a constrained nonconvex formulation)  over both directed and undirected graphs.  The experiments were performed using MATLAB R2018b on a cluster computer with two 22-cores Intel E5-2699Av4 processors (44 cores in total) and 512GB of RAM each.
The setting of our simulations is the following.

\textbf{(i) Network graph.} We simulated both undirected and directed graph, generated according to the following procedures.
\texttt{Undirected graph:}  An undirected graph is generated according
to the Erdos-Renyi model with parameter p = 0.3 (which
represents the probability of having an edge between any
two nodes). Doubly stochastic weight matrices are used, with  weights generated according to the Metropolis-Hasting rule.
\texttt{Directed graph:}  We first generate a directed cycle graph to guarantee strong connectivity.  Then we randomly add a fixed number of out-neighbors for each node.  The  row-stochastic weight matrix $W$ and the column-stochastic weight matrix $A$ are generated using uniform weights.

\textbf{(ii) Surrogate functions of \ASYSCA/ and SONATA.}  We consider two surrogate functions: $\widetilde{f}^1_i(x;x_i^k) = \nabla f_i(x_i^k)^\top (x - x_i^k) + \frac{\widetilde{\mu}}{2} \|x- x_i^k\|^2$ and  $\widetilde{f}^2_i(x;x_i^k) = \nabla f_i(x_i^k)^\top (x - x_i^k) +\frac{1}{2} (x - x_i^k)^\top H (x - x_i^k) + \frac{\widetilde{\mu}}{2} \|x- x_i^k\|^2,$ where $H$ is a diagonal matrix having the same diagonal entries as $\nabla^2 f_i(x_i^k).$  We suffix SONATA and  \ASYSCA/   with ``-L'' if the former surrogate functions are employed and with ``-DH'' if the latter are adopted.

\textbf{(iii) Asynchronous model.}  Each agent sends   its updated information to its out-neighbors and starts a new computation round, immediately after it finishes one.  The length of each computation time is sampled from a uniform distribution over the interval $[p_{\min}, p_{\max}]$.  The communication time/traveling time of each packet follows an exponential distribution $\exp(\frac{1}{D_{\text{tv}}})$.  Each agent  uses the most recent information among the arrived packets from its in-neighbors, which  in general is subject to delays.     In all our simulations, we set $p_{\min} = 5$, $p_{\max} = 15$, and $D_{\text{tv}} = 30$ (ms is the default time unit).

\textbf{(iv) Comparison with state of arts schemes.}  We compare the convergence rate of \ASYSCA/, AsyPrimalDual~\cite{wu2016decentralized} and synchronous SONATA in terms of time. The parameters are manually tuned to yield the best empirical performance for each--the used setting is reported in the caption of the associated figure.
Note that AsyPrimalDual is the only   asynchronous decentralized algorithm able to handle constraints and nonsmoothness additive functions in the objective and constraints, but only over undirected graphs and under restricted assumptions of asynchrony; also AsyPrimalDual is provably convergence only when applied to convex problems.
\vspace{-0.2cm}

\subsection{LASSO}\label{sec:simul_lasso}
\begin{figure}
		\centering
		\includegraphics[width=\linewidth]{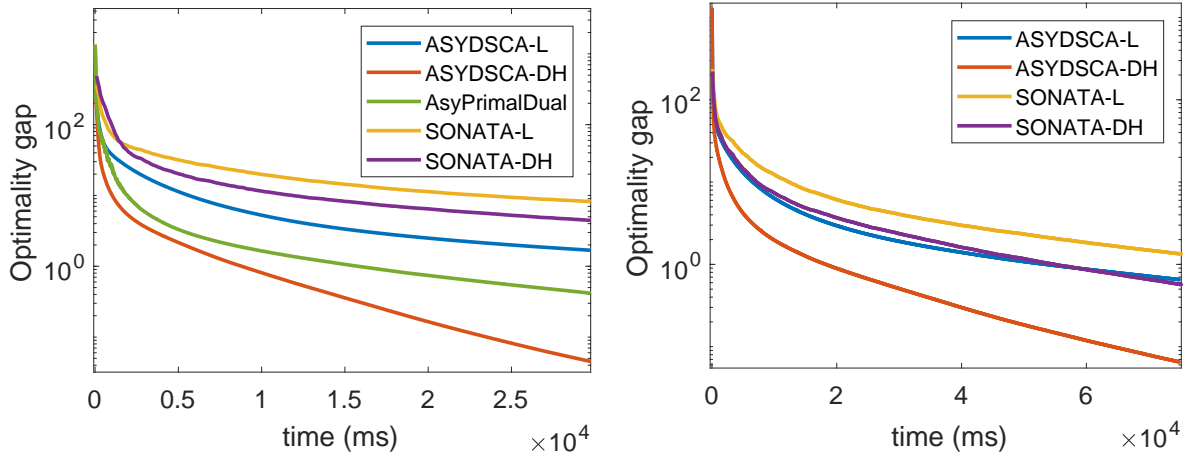}
		\captionof{figure}{\small LASSO.  \textbf{Left:} undirected graph.  We set $\tilde{\mu} = 8$ and $\gamma = 0.008$ in \ASYSCA/-L; $\tilde{\mu} = 1$ and $\gamma = 0.008$ in \ASYSCA/-DH; $\alpha = 0.06$ and $\eta = 0.6$ in AsyPrimalDual; $\tilde{\mu} = 1$ and $\gamma = 0.002$ in SONATA-L; and $\tilde{\mu} = 1$ and $\gamma = 0.005$ in SONATA-DH.  \textbf{right:} directed graph (each agent is of 10 out-neighbors).  We set $\tilde{\mu} = 10$ and $\gamma = 0.01$ in \ASYSCA/-L; $\tilde{\mu} = 10$ and $\gamma = 0.03$ in \ASYSCA/-DH;  $\tilde{\mu} = 10$ and $\gamma = 0.03$ in SONATA-L; and $\tilde{\mu} = 10$ and $\gamma = 0.05$ in SONATA-DH.}
		\label{fig:lasso}\vspace{-0.3cm}
\end{figure}

The decentralized LASSO problem reads \vspace{-0.1cm}
\begin{equation}
		\min_{x\in\mathbb{R}^n}  U(x)\triangleq    \sum_{i\in [I]} \| M_ix- b_i\|^2   +\lambda \norm{x}_1.\vspace{-0.2cm}
		\label{eq:LASSO-problem}
\end{equation}
Data $(M_i,b_i)_{i\in[I]}$ are generated as follows. We choose  $x_0\in \mathbb{R}^n$  as a ground truth sparse vector, with  $density*n$ nonzero entries drawn i.i.d. from $\mathcal{N}(0,1)$.  
Each row of $M_i \in \mathbb{R}^{r \times n}$ is drawn i.i.d. from $\mathcal{N}(0,\Sigma)$ with $\Sigma$ as a diagonal matrix such that $\Sigma_{i,i} = i^{-\omega}$.  We use $\omega$ to control the conditional number of $\Sigma.$
Then we generate $b_i =  M_i x_0 + \delta_i $, with each entry of $\delta_i$ drawn i.i.d. from $\mathcal{N}(0, 0.01)$.  We set $r = 10$, $n = 300$, $I = 20$, $\lambda = 2$, $\omega = 1.1$ and $density = 0.3$.
Since the problem satisfies the LT condition, we use $\frac{1}{I} \sum_{i\in[I]}  U\left(x_i^k \right) - U^\star$ as the optimality measure.     The result  are reported in Fig.~\ref{fig:lasso}.
\vspace{-0.3cm}

\subsection{Sparse logistic regression}
\begin{figure}
		\centering
		\includegraphics[width=\linewidth]{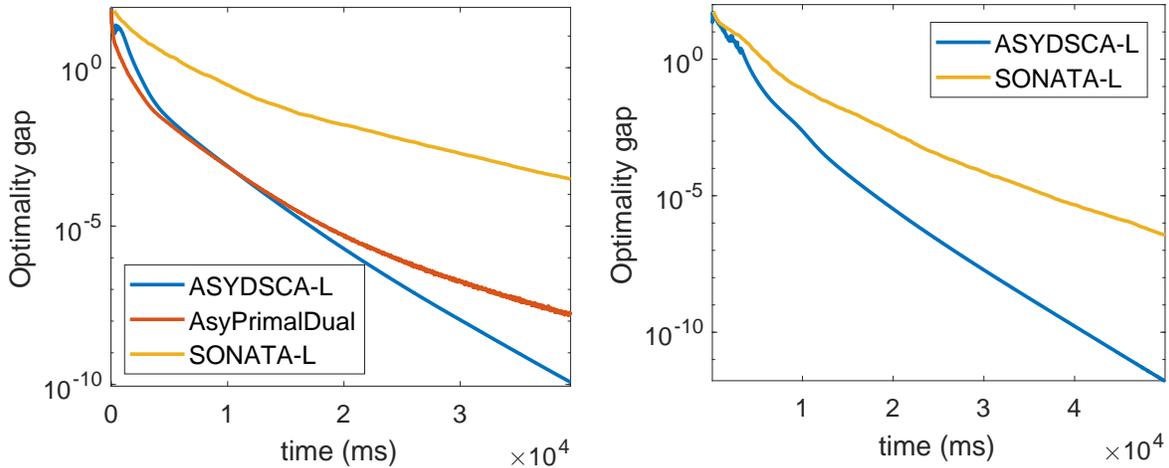}
		\captionof{figure}{\small Logistic regression.  \textbf{Left:} undirected graph.  We set $\tilde{\mu} = 10$ and $\gamma = 0.06$ in \ASYSCA/-L; $\alpha = 0.1$ and $\eta = 0.7$ in AsyPrimalDual; and $\tilde{\mu} = 10$ and $\gamma = 0.08$ in SONATA-L.  \textbf{right:} directed graph (each agent is of 10 out-neighbors).  We set $\tilde{\mu} = 10$ and $\gamma = 0.05$ in \ASYSCA/-L; and $\tilde{\mu} = 10$ and $\gamma = 0.1$ in SONATA-L.}
		\label{fig:logistic}\vspace{-0.3cm}
\end{figure}
We consider the decentralized sparse logistic regression problem in the following form
\[
\min_{x\in \mathbb{R}^n}  \sum_{i\in [I]} \sum_{s \in \mathcal{D}_i} \log (1+ \exp(-y_s \, u_s^\top x)) + \lambda \norm{x}_1,\vspace{-0.2cm}
\]
Data $(u_s,y_s)$,  ${s\in \cup_{i\in [I]}\mathcal{D}_i}$, are generated as follows. We  first choose $x_0\in \mathbb{R}^n$  as a ground truth sparse vector with  $density*n$ nonzero entries drawn i.i.d. from $\mathcal{N}(0,1)$.  We generate each sample feature $u_s$ independently, with each entry drawn i.i.d. from $\mathcal{N}(0,1)$; then we set $y_s=1$ with probability $1/(1+\exp(-u_s^\top x_0))$, and $y_s=-1$ otherwise.
We set $\abs{\mathcal{D}_i} = 3,\forall i\in [I],\, n=100,\, I =20,\,  \lambda = 0.01$ and $density=0.3.$
We use the same optimality measure as that for the LASSO problem.   The results and the tuning of parameters are reported in Fig.~\ref{fig:logistic}.\vspace{-0.2cm}

\subsection{M-estimator}

\begin{figure}
		\centering
		\includegraphics[width=\linewidth]{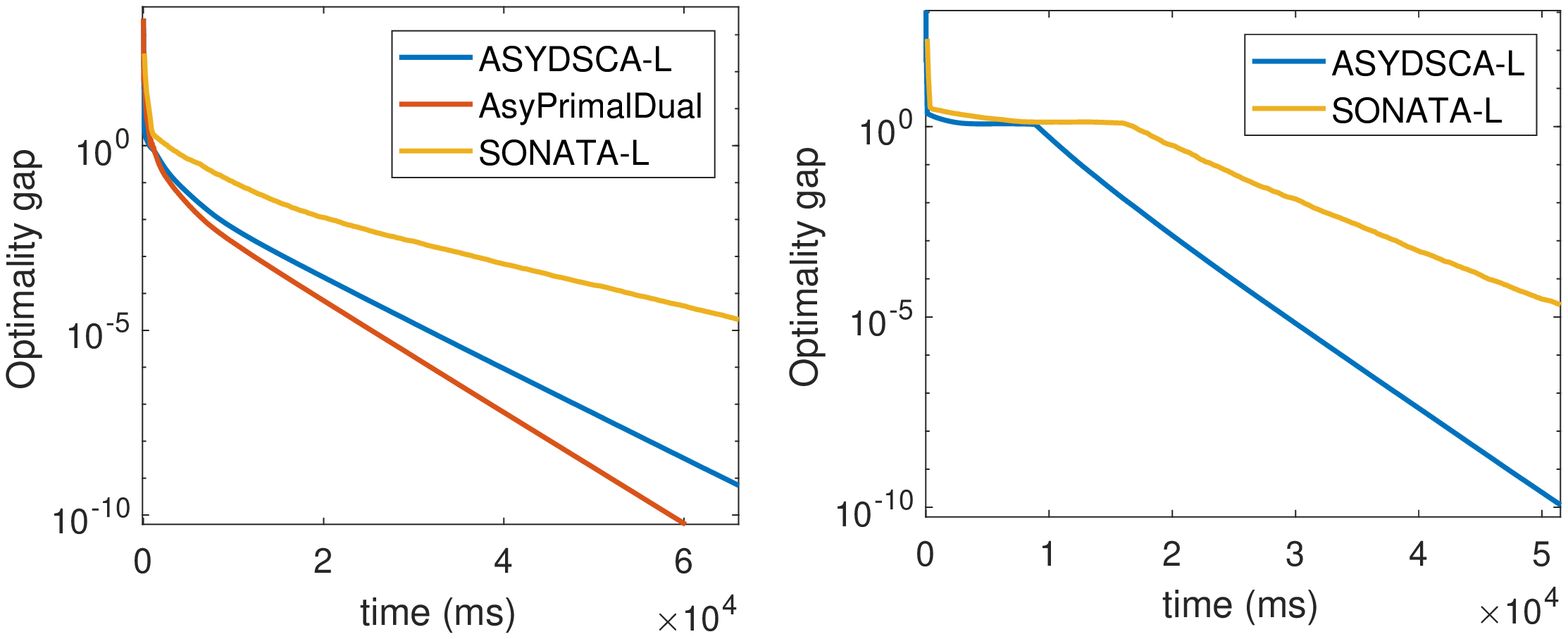}
		\captionof{figure}{\small m-estimator.  \textbf{Left:} undirected graph.  We set $\tilde{\mu} = 300$ and $\gamma = 0.1$ in \ASYSCA/-L; $\alpha = 0.01$ and $\eta = 0.6$ in AsyPrimalDual; and $\tilde{\mu} = 100$ and $\gamma = 0.1$ in SONATA-L.  \textbf{right:} directed graph (each agent is of 7 out-neighbors).  We set $\tilde{\mu} = 1000$ and $\gamma = 0.08$ in \ASYSCA/-L; and $\tilde{\mu} = 1000$ and $\gamma = 0.2$ in SONATA-L.}
		\label{fig:m-estimator}\vspace{-0.3cm}
\end{figure}

As nonconvex (constrained, nonsmooth) instance of  problem \eqref{eq:problem},  we consider the following M-estimation task   \cite[(17)]{zhang2019robustness}:
\vspace{-0.1cm}\begin{equation}
	\label{eq:M_est}
\min_{\norm{x}_2\leq r} \frac{1}{\abs{\mathcal{D}}} \sum_{i\in [I]} \sum_{s \in \mathcal{D}_i} \rho_{\alpha}(u_s^\top x - y_s) + \lambda \norm{x}_1,\vspace{-0.1cm}
\end{equation}
where $\rho_\alpha(t) = (1-e^{-\alpha \,t^2/2})/\alpha$ is the nonconvex Welsch's exponential squared loss and $\mathcal{D} \triangleq \cup_{i\in[I]} \mathcal{D}_i$.  We   generate $x_0\in \mathbb{R}^n$ as unit norm sparse vector with  $density*n$ nonzero entries drawn i.i.d. from $\mathcal{N}(0,1)$.  Each entry of $u_s \in \mathbb{R}^{n}$ is drawn i.i.d. from $\mathcal{N}(0,1)$; we generate $y_s =  {u_s}^\top x_0 + 0.1*\epsilon_s$, with $\epsilon_s \overset{i.i.d.}{\sim}\mathcal{N}(0, 1)$.
We set $\abs{\mathcal{D}_i} = 10$, for all $i\in [I]$, $n=100$,  $I =30,$  $\alpha = 0.1,$ $r =2$,  $\lambda = 0.01$, and $density=0.1.$
Since  (\ref{eq:M_est}) is nonconvex,  progresses   towards stationarity and consensus are measured using the merit function   $M_F(\cdot)$ in \eqref{eq:merit}. The result and   tuning of parameters are reported in Fig.~\ref{fig:m-estimator}.\vspace{-0.2cm}

\subsection{Discussion}\label{sec:sim-discussion}
  All the experiments clearly show that   \ASYSCA/ achieves linear rate on LASSO and Logistic regression, with nonstrongly convex objectives, both over undirected   and directed graphs--this supports our theoretical findings (Theorem \ref{thm:linear_const}). 
The flexibility in choosing the surrogate functions provides us the chance to better exploit the curvature of the objective function than plain linearization-based choices.   For example, in the LASSO experiment, ASY-DSCA-DH outperforms all the other schemes due to its advantage of better exploiting   second order information.
Also,   \ASYSCA/  compares favorably with    AsyPrimalDual.   \ASYSCA/ exhibits  good performance also in the nonconvex setting (recall that   no convergence proof is available for AsyPrimalDual applied to nonconvex problems). In our experiments,  asynchronous algorithms turned to be  faster than synchronous ones.  The reason is that,
at each iteration, agents in synchronous algorithms must  wait for the slowest agent receiving the information  and finishing its computation (no delays are allowed), before proceeding to the next iteration. This is not the case of  asynchronous algorithms wherein agents communicate and update continuously with no coordination.\vspace{-0.2cm}

\section{Proof of Theorem~\ref{thm:linear_const}}\label{sec:pf_1}
\subsection{Roadmap of the proof}\vspace{-0.1cm} We begin introducing in this section the roadmap of the proof.
Define ${x}^k \triangleq [{x}_1^k, \cdots, {x}_I^k]^{\top}$, ${v}^k \triangleq [{v}_1^k, \cdots, {v}_I^k]^{\top} \in \mathbb{R}^{I \times n}$; and let $S\triangleq (D+2)I.$ 
Construct the two  $S\times n$  matrices:\vspace{-0.2cm}  
\[\begin{aligned}
	 & \augY^k \triangleq {e}_{i^k} \left( \Delta x^k\right)^\top,\quad \text{with}\quad \Delta x^k \triangleq \widetilde{x}_{i^k}^k - {x}_{i^k}^k,\\    & \augX^{k}  \triangleq [({x}^k)^\top, ({v}^k)^\top, ({v}^{k-1})^\top, \cdots, ({v}^{k-D})^\top]^{\top},
\end{aligned}\vspace{-0.2cm}
\]
with ${v}^t=0$, for  $t \leq 0$.
Our proof builds on the  following quantities that monitor the progress of the algorithm.
\begin{subequations}\begin{itemize}
	\item Optimality gaps: \begin{equation}
			\label{def:opt_err}    	\NT^k  \triangleq  \|\tl{x}_{i^k}^k - x_{i^k}^k\|,\quad \OG^k  \triangleq  \max_{i\in [S]} U(\augX_i^k) - U^\ast; \vspace{-0.1cm}
 	\end{equation}
		\item Consensus errors ($\CV^k$ is some weighted average of row vectors of $\augX^k$ and will be defined in Sec.~\ref{sec:prelim}):\begin{equation}
		\label{def:consensus_err}    \CE^k  \triangleq  \norm{\augX^{k}- \bm{1} \cdot \CV^{k}}, \quad \TE^k \triangleq \norm{{y}_{i^k}^{k}-\avg{g}^k};\vspace{-0.1cm}
\end{equation}
\item Tracking error:\vspace{-0.1cm}\begin{equation}
	 \FY^k  \triangleq  \norm{Iy_{i^k}^k -\nabla F(x_{i^k}^k)}^2.\vspace{-0.1cm}
\end{equation}
\end{itemize}\end{subequations}
Specifically, $\NT^k$ and $\OG^k$ measure the distance of the $x_i^k$'s from optimality in terms of  step-length and objective value. $\CE^k$ and $\TE^k$ represents the consensus error of  $x_i$'s and $y_i$'s, respectively while  $\FY^k$ is the tracking error of $y_i^k$.  Our goal is to show that the above quantities vanish at a linear rate, implying  convergence (at the same rate)  of the iterates generated by the algorithm to a solution of  Problem~\eqref{eq:problem}.    Since  each of them affects the dynamics of the others, our proof begins   establishing the following  set of inequalities linking these quantities  (the explicit expression of the constants below will be given in the forthcoming sections):\vspace{-0.2cm}
\begin{subequations}\label{eq:small_gain_inequalities}
	\begin{align}
 	\begin{split}
			& \label{eq:consensus_bound_y}\TE^{k+1}   \leq  3 C_1 l \sum_{l=0}^k \rho^{k-l}  \left(  \CE^l  + \gamma \NT^l \right)    + C_1 \rho^k \norm{{g}\spe{0}},
		\end{split}\\
		\begin{split}
			& \label{eq:consensus_bound_x}\CE^{k+1}  \leq  C_2 \rho^k \CE^{0} + C_2 \sum_{l=0}^k \rho^{k-l} \gamma \NT^l,
		\end{split}\\
		\begin{split}
			& \label{eq:tracking_bound}\FY^k  \leq  8 I\,l^2 (\CE^k)^2 + 2 I^2 (\TE^k)^2,
		\end{split}\\
		\begin{split}
			& \label{eq:opt_bound}\OG^{k+1}
			\leq  C_4(\gamma)  \, \zeta(\gamma)^{k} \OG^{0} + \frac{C_3(\gamma) C_4(\gamma)}{\zeta(\gamma)} \sum_{\ell = 0}^k  \zeta(\gamma)^{k-\ell} \FY^{\ell},
		\end{split}\\
		\begin{split}
			& \label{eq:dx_bound}(\NT^k )^2  \leq \frac{1}{\gamma \left( \tilde{\mu} - \frac{\epsilon}{2}- \frac{\gamma L}{2}  \right) } \OG^k + \frac{1}{2  \, \epsilon\left( \tilde{\mu} - \frac{ \epsilon}{2}- \frac{\gamma L}{2}  \right)  } \FY^k.
		\end{split}
	\end{align}
\end{subequations}
We then show that  $\NT^k$,  $\OG^k$, $\CE^k$, $\TE^k$ and $\TE^k$ vanish at linear rate chaining the above inequalities by means of the generalized small gain theorem \cite{Tian_arxiv}.

The main steps of the proof are summarized next.

 \noindent $\bullet$ \textbf{Step 1: Proof of \eqref{eq:consensus_bound_y}-\eqref{eq:tracking_bound} via P-ASY-SUM-PUSH.}
We rewrite    (S.2)  and  (S.3) in \ASYSCA/ (Algorithm~\ref{alg:ASYSCA}) as instances of  the perturbed asynchronous consensus scheme and the perturbed asynchronous sum-push scheme (the P-ASY-SUM-PUSH) introduced in the companion paper \cite{Tian_arxiv}. By doing so,  we can bound  the consensus errors $\CE^k$ and $\TE^k$ in terms of $\Delta^k$ and then prove \eqref{eq:consensus_bound_y}-\eqref{eq:consensus_bound_x}--see Lemma~\ref{lm:consensus_err} and  Lemma~\ref{lm:tracking_err}. Eq.~\eqref{eq:tracking_bound} follows readily from \eqref{eq:consensus_bound_y}-\eqref{eq:consensus_bound_x}--see Lemma \ref{lm:sqr_sum}.  

\noindent $\bullet$ \textbf{Step 2: Proof of \eqref{eq:opt_bound}-\eqref{eq:dx_bound}  under  the LT condition.} Proving   \eqref{eq:opt_bound}--contraction of the  optimality measure $\OG^k$ up to the tracking error--poses several challenges.  To  prove 
  contraction of some form of optimization errors, existing techniques developed in the literature of {\it distributed} algorithms   \cite{shi2015extra,Nedich-geometric,qu2017harnessing,sun2019convergence,alghunaim2019decentralized}   (including   our companion paper \cite{Tian_arxiv}) leverage  strong convexity of $F$, a property that is replaced here by the weaker local growing condition  \eqref{eq:LT_effect} in the LT error bound. 
  Hence, they are not applicable to our setting.
 On the other hand,  existing proofs showing linear rate of {\it centralized} first-order methods under the LT condition  \cite{luo1993error} do not readily customize to our distributed, asynchronous setting, for the  reasons elaborated next. To invoke the local growing condition  \eqref{eq:LT_effect},  one needs   first to show that the sequences generated by the algorithm enters (and stays into) the region where    \eqref{eq:LT_condition} holds, namely: a) the function value remains  bounded; and b) the proximal operator residual is sufficiently small. A standard path to prove a) and b) in the centralized setting is showing that the objective function sufficiently descents along the trajectory of the algorithm.   Asynchrony apart,  in the distributed setting, function values  on the agents' iterates  do not   monotonically decrease provably,   due to  consensus and gradient tracking errors. To cope with these issues, in this Step 2, we put forth a new analysis. Specifically,   i) Sec.~\ref{sec:asymptotic_convergence}: we   build  a novel   Lyapunov function [cf. \eqref{eq:Ly-convex}] that linearly combines   objective values of current and past (up to $D$) iterates (all the elements of $H^k$); notice that the choice of the weights  (cf.  ${\psi}^{k}$ in Lemma \ref{lm:Wscramb}) is very peculiar and represents a major departure from existing approaches (including our companion paper \cite{Tian_arxiv})--${\psi}^{k}$ {\it endogenously vary    according to the asynchrony trajectory of the algorithm}. The Lyapunov function  is proved to ``sufficiently'' descent over the asynchronous iterates of \ASYSCA/ (cf. Proposition \ref{thm:lyapu}); ii)  Sec.~\ref{sec:err_bound}:   building on  such descent properties, we manage to prove that   $x_{i^k}^k$  will eventually satisfy  the aforementioned conditions  \eqref{eq:LT_condition} (cf. Lemma \ref{lm:resi} \& Corollary \ref{cor:bdd_err}), so that the LT   growing property \eqref{eq:LT_effect} can be invoked  at $x_{i^k}^k$ (cf. Corollary~\ref{cor:err_bound}); iii) Sec. \ref{sec:opt_err}: Finally, leveraging this local   growth, we  uncover relations between $\OG^k$  and  $\FY^k$ and  prove \eqref{eq:opt_bound} (cf. Proposition \ref{prop:OG}). Eq. \eqref{eq:dx_bound} is proved in   Sec. \ref{sec:opt_err} by product of the  derivations above.

\noindent $\bullet$ \textbf{Step 3: R-linear convergence via the generalized small gain theorem.}
We complete the proof of linear convergence by applying   \cite[Th. 23]{Tian_arxiv} to the inequality system \eqref{eq:small_gain_inequalities}, and conclude that all the local variables $\{x_i\}_{i\in[I]}$ converge to the set of optimal solutions $\KK^\ast$ R-linearly.\vspace{-0.3cm}

\subsection{Step 1: Proof of \eqref{eq:consensus_bound_y}-\eqref{eq:tracking_bound}}\label{sec:prelim} 

We interpret the consensus step (S.2) in Algorithm 2  as an instance of the perturbed asynchronous consensus scheme~\cite{Tian_arxiv}: (\ref{eq:asy-consensus}) can be rewritten as
\begin{align}\label{eq:chi}
	{\augX}^{k+1} = \W^k({\augX}^{k} + \gamma \augY^k),
\end{align}
where $\W^k$ is a time-varying augmented matrix induced by the update order of the agents and the delay profile.    The specific expression of $\W^k$ can be found in \cite{Tian_arxiv} and is omitted here, as it is  not  relevant to the convergence proof. We only need to recall the following properties of $\W^k$.

\begin{lemma}\label{lm:Wscramb} \cite[Lemma~17]{Tian_arxiv}
	Let $\{\W^k \}_{k\in \mathbb{N}_+}$ be the sequence of matrices in the dynamical system \eqref{eq:chi}, generated under Assumption~\ref{ass:delays}, and with   ${W}$ satisfying Assumption~\ref{assumption_weights} (i), (ii).  Define $K_1 \triangleq (2I-1) \cdot T+I \cdot D $, $C_2 \triangleq \frac{2\sqrt{(D+2)I}  (1+ \lbm^{-K_1})}{1- \lbm^{-K_1}}$, $\eta \triangleq \bar{m}^{K_1}$ and $\rho \triangleq (1-\eta)^{\frac{1}{K_1}}$.  Then we have for any $k\geq 0$:
	\begin{enumerate}[(i)]
		\item $\W^k$ is row stochastic;
		\item all the entries in the first $I$ columns of $\W^{k+K_1-1:k}$ are uniformly bounded below by $\eta  ;$
		\item  there exists a sequence of stochastic vectors $\{{\psi}^k\}_{k\geq 0}$ such that: i) for any $\ell \geq t\geq 0$, $\norm{\W^{\ell:t} -\bd{1}{{\psi}^t}^{\top}}_2 \leq C_2 \rho^{\ell-t}$;  ii) $\psi_i^k\geq \lbpsi  \,$ for all $i\in \mathcal{V}$.
	\end{enumerate}
\end{lemma}

Note that  Lemma~\ref{lm:Wscramb} implies
\begin{equation}\label{eq:prod_mat}
	\hspace{-0.3cm}\bd{1}{{\psi}^t}^{\top}=  \lim_{n\to\infty} \W^{n:t} = (\lim_{n\to\infty} \W^{n:t+1} )\W\spe{t} =  \bd{1} {{\psi}^{t+1}}^{\top} \W\spe{t},
\end{equation}
and thus
$
{{\psi}^{t+1}}^{\top} \W\spe{t} = {{\psi}^t}^{\top}$, for all $t \geq 0.$
Then we define 
\begin{equation}\vspace{-0.2cm}
	\CV^{k} = {{\psi}^k}^{\top}{\augX}^k;
\end{equation}
  $\CV^{k}$ evolves according to the following dynamics:\vspace{-0.2cm}
\begin{equation}\label{eq:avg_dynamic_x}
	\CV^{k+1} =  {\bd{{\psi}}^0}^{\top}{\augX}^0+\sum_{l=0}^k  \gamma{\bd{{\psi}}^l}^{\top}\Delta  {\augX}^l.\vspace{-0.2cm}
\end{equation}
This can be shown by applying \eqref{eq:chi} recursively, so that
\begin{equation}\label{eq:x_evolution}\vspace{-0.2cm}
	{\augX}^{k+1} = \W^{n:0}{\augX}^0+\sum_{l=0}^k \W^{n:l}\gamma \Delta  {\augX}^l,
\end{equation}
and multiplying \eqref{eq:x_evolution} from the left by ${{\psi}^{k+1}}^{\top}$ and using~\eqref{eq:prod_mat}.  Taking the difference between~\eqref{eq:avg_dynamic_x} and~\eqref{eq:x_evolution} and applying Lemma~\ref{lm:Wscramb}  the consensus error $\CE^k$ can be bound as follows.

\begin{lemma}\label{lm:consensus_err}
	Under the condition of Lemma~\ref{lm:Wscramb},  $\{\CE^{k}\}$ satisfies\vspace{-0.2cm}
	\begin{equation}\label{lm:inter_eq1}
		\CE^{k+1} \leq  C_2 \, \rho^k \, \CE^{0} + C_2 \sum_{l=0}^k \rho^{k-l} \gamma \, \NT^l,\quad \forall k\geq 0. \vspace{-0.2cm}
	\end{equation}
\end{lemma}

To establish similar bounds for $E_y^k$,  we build on the fact that  the  gradient tracking update (\ref{eq:asy-track}) is an instance of the P-ASY-SUM-PUSH   in \cite{Tian_arxiv}, as shown next.
Define
\begin{align*}
&	g^k = [\nabla f_1(x_1^k), \nabla f_2(x_2^k), \cdots, \nabla f_I(x_I^k)]^\top, \\
&	\avg{g}^k = (1/I) \cdot  (g^k )^\top  \mathbf{1}, \qquad \TE^k \triangleq \norm{{y}_{i^k}^{k}-\avg{g}^k}.
\end{align*}
We can   prove the following bound for $\TE^{k}$.
\begin{lemma}\label{lm:tracking_err}
	Let $\{x^k, y_{i^k}^k\}_{k=0}^{\infty}$ be the sequence generated by the Algorithm~\ref{alg:ASYSCA} under Assumption~\ref{ass:strong_conn},~\ref{assumption_weights}, and~\ref{ass:delays}.  Then, there exists a constant $C_1=\frac{4\sqrt{2S}(1+\lbm^{-K_1})}{I\,\eta\,\rho(1-\lbm^{K_1})}$ such that \vspace{-0.1cm}
	\begin{equation}\label{lm:inter_eq2}
		\TE^{k+1}  \leq  3 \, C_1 \, l \sum_{l=0}^k \rho^{k-l}  \left(  \CE^l  + \gamma  \NT^l \right)    + C_1 \rho^k \norm{{g}\spe{0}}.\vspace{-0.1cm}
	\end{equation}

\end{lemma}

\begin{proof}
See Appendix~\ref{ap:lm:tracking_err}.\end{proof}

 Finally, using Lemma~\ref{lm:consensus_err} and Lemma~\ref{lm:tracking_err}, we can bound  $\sum_{t=0}^k (\CE^t)^2$ and $\sum_{t=0}^k (\TE^t)^2$ in terms of $\sum_{t=0}^{k} \gamma^2 (\NT^t)^2$,  and  $\FY^k$ in terms of $\CE^k$ and $\TE^k$, as given below.
\begin{lemma}\label{lm:sqr_sum}
	Under the setting of Lemma~\ref{lm:consensus_err} and Lemma~\ref{lm:tracking_err}, we have: for any $k \geq 1$,\vspace{-0.2cm}
	\begin{align}
	&  \sum_{t=0}^k (\CE^t)^2 \leq c_{x} + \varrho_{x} \sum_{t=0}^{k} \gamma^2 (\NT^t)^2 , \nonumber \\
	&  \sum_{t=0}^k (\TE^t)^2 \leq   c_{y} + \varrho_{y} \sum_{t=0}^{k} \gamma^2 (\NT^t)^2,   \nonumber \\
	&\label{eq:yct}  \FY^k  \leq   2 I^2 (\TE^k)^2 + 8 I\,l^2\, (\CE^k)^2.
	\end{align}
	with $\varrho_{x} \triangleq \frac{2C_2^2}{(1-\rho)^2},$ and   $\varrho_{y} \triangleq \frac{36 \left( C_1 L  \right)^2 \left(2 C_2^2 + (1-\rho)^2 \right)}{(1-\rho)^4}.$
	(The expressions of the constants $c_{x}$ and $c_{y}$ are omitted as they are not relevant).
\end{lemma}
\begin{proof}
The proof of the first two results follows similar steps as in that of \cite[Lemma~26]{Tian_arxiv} and thus is omitted. We prove only the last inequality, as follows:
\begin{equation*}
		\begin{aligned}
			  & \FY^k  = \norm{I y_{i^k}^k \pm I \bar{g}^k- \nabla F (x_{i^k}^k)}^2\\
			& \leq 2 I^2 (\TE^k)^2 + 2 \Big\|\sum_{j=1}^I f_j (x_j^k) \pm F (\CV^k)-  \nabla F (x_{i^k}^k)\Big\|^2\\
			& \leq   2 I^2 (\TE^k)^2 + 8 I\,l^2\, (\CE^k)^2  .
		\end{aligned}\vspace{-0.6cm}
	\end{equation*}
\end{proof}

\subsection{Step 2: Proof of \eqref{eq:opt_bound}-\eqref{eq:dx_bound}  under  the LT condition}

\subsubsection{A new Lyapunov function and its descent}\label{sec:asymptotic_convergence}
We begin studying descent of the objective function $U$ along the trajectory of the algorithm; we have the following result. \begin{lemma}\label{lem:v_descent_1}
	Let $\{(x^k,y^k)\}$ be the sequence generated by Algorithm~\ref{alg:ASYSCA} under Assumptions~\ref{ass:cost_functions} and~\ref{assump:surrogate}, it holds
	\begin{align}\label{eq:v_descent_1}
		U(v_{i^k}^{k+1})  \leq & \,U(x_{i^k}^k)-  \gamma \left( \tilde{\mu} - \frac{\gamma L}{2} \right) \norm{\dx^k}^2  \nonumber \\
		&+  \gamma \cdot \left(\nabla F(x_{i^k}^{k}) - I y_{i^k}^k  \right)^\top \dx^k.
	\end{align}
\end{lemma}

\begin{proof}
	Applying the first order optimality condition to~\eqref{eq:sca} and invoking the strong convexity of $\tf_{i^k}$ (Assumption~\ref{assump:surrogate})
	we have\vspace{-0.1cm}
	\begin{equation}\label{eq:mini_prin}
	\begin{aligned}
	&	 - (\dx^k)^{\top} Iy_{i^k}^k + G(x_{i^k}^k) - G(\tl{x}_{i^k}^k) \\
	&	 \geq - (\dx^k)^{\top} (\nabla f_{i^k}(x_{i^k}^k)-\nabla \widetilde{f}_{i^k}(\tl{x}_{i^k}^k;x_{i^k}^k))  \\
	&	=  (\dx^k)^{\top} \big(\nabla \widetilde{f}_{i^k}(\tl{x}_{i^k}^k;x_{i^k}^k) - \nabla \widetilde{f}_{i^k}(x_{i^k}^k;x_{i^k}^k) \big) \geq \tilde{\mu} \cdot \norm{\dx^k}^2.
	\end{aligned}
	\end{equation}
	As $F$ is $L$-smooth, applying the descent lemma gives\vspace{-0.2cm}
	\begin{align*}
		\begin{split}
			& F(v_{i^k}^{k+1})  \leq F(x_{i^k}^{k}) + \gamma \cdot \nabla F(x_{i^k}^{k})^\top  \dx^k + \frac{L}{2} \gamma^2 \norm{\dx^k}^2 \\
			& =  F(x_{i^k}^{k}) + \gamma \cdot (I y_{i^k}^k )^\top  \dx^k  + \gamma \cdot \left(\nabla F(x_{i^k}^{k}) - I y_{i^k}^k  \right)^\top   \dx^k   + \frac{L}{2} \gamma^2 \norm{\dx^k}^2 \\
			& \overset{\eqref{eq:mini_prin}}{\leq}  F(x_{i^k}^{k}) + \gamma \left( G(x_{i^k}^k) - G(\tl{x}_{i^k}^k) - \tilde{\mu} \norm{\dx^k}^2 \right)   + \frac{L}{2} \gamma^2 \norm{\dx^k}^2  + \gamma \cdot \left(\nabla F(x_{i^k}^{k}) - I y_{i^k}^k  \right)^\top \dx^k.
		\end{split}
	\end{align*}
	 By the convexity of $G$, we have
	\begin{equation*}
	\gamma \left( G(x_{i^k}^k) - G(\tl{x}_{i^k}^k)  \right)  \leq G(x_{i^k}^k) - G(v_{i^k}^{k+1}).
	\end{equation*}
	Combining the above two results proves~\eqref{eq:v_descent_1}.\vspace{-0.2cm}
\end{proof}

We build now on (\ref{eq:v_descent_1}) and establish descent on  a suitable defined Lyapunov function.
Define the mapping $ \widetilde{U}: \mathbb{R}^{S\times n} \to \mathbb{R}^S$ as $\widetilde{U}(\augX)=[U(\augx_1),\cdots,U(\augx_S)]^{\top}$ for $\augX = [\augx_1,\cdots, \augx_S]^\top \in \mathbb{R}^{S \times n}$.  That is, $\widetilde{U}(\augX)$  is a vector constructed by stacking the value of the objective function $U$ evaluated at each local variable $h_i$. Recalling the definition of the weights ${\psi}^{k}$ (cf. Lemma \ref{lm:Wscramb}), we introduce the Lyapunov function \vspace{-0.2cm}\begin{equation}\label{eq:Ly-convex}
	\Ly/^k\triangleq {{\psi}^{k}}^{\top}  \widetilde{U} (\augX^k),\vspace{-0.2cm}
\end{equation}
and study next its descent properties. 

\begin{proposition}\label{thm:lyapu}
	Let $\{(x^k,v^k,y^k)\}$ be the sequence generated by Algorithm~\ref{alg:ASYSCA} under Assumptions~\ref{ass:cost_functions}, \ref{ass:strong_conn},~\ref{assump:surrogate}, and  \ref{assumption_weights}.  Then,
	\begin{equation}\label{eq:lyapu}
		\begin{aligned}
		&	\Ly/^{k+1}
			\leq \Ly/^0 \\
		&  ~- \sum_{t=0}^k (\NT^t)^2  \gamma \left( \eta \tilde{\mu} - \gamma \left( \frac{L}{2}   +   l \, I^{\frac{3}{2}}\sqrt{ \varrho_{x}}  +  I \sqrt{\varrho_{y}}  \right)\right) + C,
		\end{aligned}
	\end{equation}
	for all $k\geq 0$, where $C$ is some constant independent of $\gamma$ and $k$; and   $\varrho_{x}$ and $\varrho_{y}$ are defined in Lemma~\ref{lm:sqr_sum}.
\end{proposition}

\begin{proof}
	By the row stochasticity of $\W$ and the convexity of $U$:
	\begin{align*}
		& \widetilde{U}(\augX^{k+1})  =  \widetilde{U} \left(\W^k (\augX^k + \gamma\Delta {\augX}^k )\right)  \preccurlyeq \W^k\, \widetilde{U} \left(\augX^k + \gamma\Delta {\augX}^k \right) \\
		& \preccurlyeq \W^k \bigg( \widetilde{U} (\augX^k) -  \bigg( \gamma \bigg( \tilde{\mu} - \frac{\gamma L}{2} \bigg)  \norm{\dx^k}^2  - \gamma \cdot  \big(\nabla F(x_{i^k}^{k}) - I y_{i^k}^k  \big)^\top  \dx^k\bigg) {e}_{i^k} \bigg) ,
	\end{align*}
	where  in the last inequality
 we  applied Lemma~\ref{lem:v_descent_1}.
	Using now Lemma~\ref{lm:Wscramb}, we have
	\begin{align}
		\begin{split}
			&\Ly/^{k+1} \\
			&\leq \Ly/^k - \psi_{i^k}^{k} \bigg( \gamma \bigg( \tilde{\mu} - \frac{\gamma L}{2} \bigg)  \norm{\dx^k}^2  \!\!\!
			  -  \gamma \left(\nabla F(x_{i^k}^{k}) - I y_{i^k}^k  \right)^\top \!\!\!\dx^k \bigg) \nonumber \\
			\label{eq:Lyapu1}& \leq \Ly/^k -  \gamma \eta \tilde{\mu}  \norm{\dx^k}^2 + \frac{L(\gamma)^2}{2} \norm{\dx^k}^2  + \psi_{i^k}^{k} \,\gamma \left(\nabla F(x_{i^k}^{k}) - I y_{i^k}^k  \right)^\top \dx^k\\
			&  \leq  \Ly/^k -  \gamma \left( \eta \tilde{\mu} - \frac{\gamma L}{2} \right) \norm{\dx^k}^2  \\
			  &\quad +   \psi_{i^k}^{k} \gamma  \left( \nabla F( x_{i^k}^k) \pm I\avg{g}^k - I \,y_{i^k}^k  \right)^\top  \dx^k   \\
			& \leq \Ly/^k -  \gamma \left( \eta \tilde{\mu} - \frac{\gamma L}{2} \right) \norm{\dx^k}^2    +   \gamma\, I \, l \sum_{j=1}^I  \norm{\CV^{k}-x_j^k} \norm{\dx^k}   + \gamma\, I  \TE^k \, \cdot \norm{\dx^k} \\
			& \leq  \Ly/^k -  \gamma \left( \eta \tilde{\mu} - \frac{\gamma L}{2} \right)\norm{\dx^k}^2     +  \gamma\, l \, I^{\frac{3}{2}}\CE^k \, \norm{\dx^k}   + \gamma\, I  \TE^k \, \cdot \norm{\dx^k}^2 \\
			& \overset{(*)}{\leq} \Ly/^k -  \gamma \left( \eta \tilde{\mu} - \gamma \left( \frac{ L}{2}+ \frac{1}{2\epsilon_1} + \frac{1}{2\epsilon_2}  \right) \right) \norm{\dx^k}^2      +  \frac{\epsilon_1}{2}\, l^2 \,I^3 (\CE^k)^2    + \frac{\epsilon_2}{2}\, I^2 ( \TE^k)^2  \\
			& \leq \Ly/^0 -  \gamma \left( \eta \tilde{\mu} - \gamma \left( \frac{ L}{2}+ \frac{1}{2\epsilon_1} + \frac{1}{2\epsilon_2}  \right) \right) \sum_{t=0}^k \norm{\dx^t}^2     +  \frac{\epsilon_1}{2}\, l^2 \,I^3 \sum_{t=0}^k (\CE^t)^2    + \frac{\epsilon_2}{2}\, I^2  \sum_{t=0}^k ( \TE^t)^2,
		\end{split}
	\end{align}
	where   ($*$) follows from the Young's inequality with $\epsilon_{1,2} >0$.
	Invoking Lemma~\ref{lm:sqr_sum} and setting $\gamma^l \equiv \gamma$ gives~\eqref{eq:lyapu}, where  the free parameters $\epsilon_{1,2}$ are chosen as $\epsilon_1 = {1}/({l \, I^{\frac{3}{2}}\sqrt{  \varrho_{x}}})$ and $ \epsilon_2 =  {1}/({I \sqrt{\varrho_{y}}})$, respectively.
\end{proof}

\subsubsection{Leveraging the LT condition}\label{sec:err_bound}
We build now on Proposition \ref{thm:lyapu} and  show next that the two  conditions  in (\ref{eq:LT_condition}) holds at $x_i^k$, for sufficiently large $k$; this will permit to invoke the LT growing property \eqref{eq:LT_effect}.

The first condition--$U(x_i^k)$ bounded for large $k$--is a direct consequence of Proposition \ref{thm:lyapu} and the facts that  $U$ is bounded from below (Assumption~\ref{ass:cost_functions})  and  $\psi_{i}^k  \geq \eta$, for all $i \in [I]$ and $k \geq 0$. Formally, we have the following.
\begin{corollary}\label{cor:bdd_err}
	Under the setting of Proposition~\ref{thm:lyapu} and  step-size   $0< \gamma < \bar{\gamma} \triangleq \frac{2\eta \tilde{\mu}}{L + 2l\, I^{\frac{3}{2}}\sqrt{  \varrho_{x}}+ 2I \sqrt{\varrho_{y}} }$,  it holds:
	\begin{enumerate}[(i)]
		\item $U(x_i^k) $ is uniformly upper bounded,  for all $i\in \mathcal{V}$ and $k \geq 0$;
		\item $\sum_{t=0}^\infty (\NT^t)^2 < \infty$, $\sum_{t=0}^\infty (\CE^t)^2 < \infty$, and $\sum_{t=0}^\infty (\TE^t)^2 < \infty$.
	\end{enumerate}
\end{corollary}



We prove now that the second condition in (\ref{eq:LT_condition}) holds for large $k$--the residual of the proximal operator at $x_{i^k}^k$, that is  $ \norm{x_{i^k}^k - \text{prox}_G (x_{i^k}^k -\nabla F(x_{i^k}^k))}$,  is sufficiently small. Since $\NT^k$ and  the gradient tracking error $\FY^k$ are vanishing [as a consequence of  Corollary \ref{cor:bdd_err}ii) and Lemma \ref{lm:sqr_sum}], it is sufficient to bound  the aforementioned residual by $\NT^k$ and  $\FY^k$. This is done in the lemma below.

\begin{lemma}\label{lm:resi}
	The proximal operator residual on $x_{i^k}^k$ satisfies
	$\norm{x_{i^k}^k - \text{prox}_G (x_{i^k}^k -\nabla F(x_{i^k}^k))}^2  \leq  4 \big( 1 + (l + \tilde{l})^2 \big) (\NT^k)^2 + 5 \FY^k.$
\end{lemma}
\begin{proof}
	For simplicity, we denote $\hat{x}^k = \text{prox}_G (x_{i^k}^k -\nabla F(x_{i^k}^k))$.  According to the variational characterization of the proximal operator, we have, for all $ w \in \mathcal{K}$,\vspace{-0.1cm}
	\[
	\left( \hat{x}^k - \big(x_{i^k}^k -\nabla F(x_{i^k}^k) \big)  \right)^\top (\hat{x}^k  - w) + G(\hat{x}^k ) -G(w)  \leq 0.
	\]
	The first order optimality condition of $\tl{x}_{i^k}^k$ implies\vspace{-0.1cm}
	\begin{align}\label{eq:FOC_general}
		& \left( \nabla \widetilde{f}_{i^k}(\tl{x}_{i^k}^k;x_{i^k}^k)+Iy_{i^k}^k-\nabla f_{i^k}(x_{i^k}^k) \right)^\top (\tl{x}_{i^k}^k - z)   + G(\tl{x}_{i^k}^k ) -G(z)  \leq 0, \quad \forall z \in \mathcal{K}.
	\end{align}

	Setting $z = \hat{x}^k $ and $w = \tl{x}_{i^k}^k $ and adding the above two inequalities  yields
	\begin{align*}
		& 0 \geq  \bigg( \nabla \widetilde{f}_{i^k}(\tl{x}_{i^k}^k;x_{i^k}^k)+Iy_{i^k}^k-\nabla f_{i^k}(x_{i^k}^k) - \hat{x}^k   + x_{i^k}^k -\nabla F(x_{i^k}^k)  \bigg)^\top (\tl{x}_{i^k}^k - \hat{x}^k)\\ 
			& =  \left( Iy_{i^k}^k - \hat{x}^k + x_{i^k}^k -\nabla F(x_{i^k}^k)  \right)^\top (\tl{x}_{i^k}^k - x_{i^k}^k)   + \left( \nabla \widetilde{f}_{i^k}(\tl{x}_{i^k}^k;x_{i^k}^k) - \nabla f_{i^k}(x_{i^k}^k)  \right)^\top  (\tl{x}_{i^k}^k - x_{i^k}^k)  \\
		& ~~ + \norm{\hat{x}^k - x_{i^k}^k}^2 + \bigg( \nabla \widetilde{f}_{i^k}(\tl{x}_{i^k}^k;x_{i^k}^k)+Iy_{i^k}^k-\nabla f_{i^k}(x_{i^k}^k)    -\nabla F(x_{i^k}^k)  \bigg)^\top (x_{i^k}^k - \hat{x}^k)  \\
		& \geq  -\frac{1}{2}\norm{ Iy_{i^k}^k -\nabla F(x_{i^k}^k)  }^2  -\frac{1}{2} \norm{\dx^k}^2  - \frac{1}{4} \norm{\hat{x}^k - x_{i^k}^k}^2 \\
		& ~~ - \norm{\dx^k}^2 + \widetilde{\mu} \norm{\dx^k}^2 + \norm{\hat{x}^k - x_{i^k}^k}^2  - \frac{1}{4} \norm{\hat{x}^k - x_{i^k}^k}^2   \\
		& ~~ -   2 \left( (l+\widetilde{l})^2 \norm{\dx^k}^2 +\norm{Iy_{i^k}^k -\nabla F(x_{i^k}^k) }^2 \right). 	\end{align*}
	Rearranging terms proves the desired result.
\end{proof}
Corollary \ref{cor:bdd_err} in conjunction with  Lemma \ref{lm:resi} and Lemma \ref{lm:sqr_sum} show that both conditions in (\ref{eq:LT_condition}) hold at $\{x^k\}$, for large $k$. We can then invoke the growing condition (\ref{eq:LT_effect}).
\begin{corollary}\label{cor:err_bound}
	Let $\{x^k\}$ be the sequence generated by Algorithm~\ref{alg:ASYSCA} under the setting of Corollary~\ref{cor:bdd_err}. Then, there exists a constant $\kappa>0$ and a sufficiently large  $\bar{k}$ such that, for $k \geq \bar{k}$,
	\begin{align}\label{eq:err_bou}
		\text{dist}(x_{i^k}^k, \mathcal{K}^*) \leq \kappa \, \norm{x_{i^k}^k - \text{prox}_G (x_{i^k}^k -\nabla F(x_{i^k}^k))} .
	\end{align}
\end{corollary}
\begin{proof}
	It is sufficient to show that \eqref{eq:LT_condition} holds at $x_{i^k}^k$.  By Corollary~\ref{cor:bdd_err}(i),   $U(x_{i^k}^k) \leq B$, for all $k \geq 0$ and some $B < +\infty$.  Lemma~\ref{lm:resi} in conjunction with Corollary~\ref{cor:bdd_err}(ii) and Lemma \ref{lm:sqr_sum} yields	 $\lim\limits_{k\to\infty} \norm{x_{i^k}^k - \text{prox}_G (x_{i^k}^k -\nabla F(x_{i^k}^k))} = 0$. 
\end{proof}
\vspace{-0.2cm}

\subsubsection{Proof of \eqref{eq:opt_bound}}\label{sec:opt_err}
Define
\begin{align}\label{eq:C_3}
	& C_3(\gamma)  \triangleq \frac{ \gamma \left( c_6  (\tilde{\mu} - \frac{ \epsilon}{2}- \frac{\gamma L}{2} ) +  \frac{c_7}{2 \epsilon} \right)}{c_7 +  \tilde{\mu} - \frac{ \epsilon}{2}- \frac{\gamma L}{2} }, \\
	& C_4(\gamma) \triangleq \left( 1 - \Big(1- \sigma(\gamma) \Big) \eta \right)^{-1},  \\
	& \zeta(\gamma) \triangleq \left( 1 - \Big(1- \sigma(\gamma) \Big) \eta \right)^{\frac{1}{K_1}}, \\
	&  \sigma( \gamma)   \triangleq \frac{c_7 +  \left( \tilde{\mu} - \frac{  \epsilon}{2}- \frac{\gamma L}{2} \right) (1-\gamma)}{ c_7 + \tilde{\mu} - \frac{ \epsilon}{2}- \frac{\gamma L}{2}},
\end{align}
where $K_1 = (2I-1) \cdot T + I \cdot D$, and $c_6,\,c_{7}$ are polynomials in $(1,l,\tilde{l},L,\kappa)$ whose expressions are given in~\eqref{eq:const_expression} and~\eqref{eq:ik_desc4}; and $\epsilon \in (0, 2\tilde{\mu})$ is a free parameter (to be chosen).

In this section, we prove \eqref{eq:opt_bound}, which is formally stated in the  proposition below.
\begin{proposition}\label{prop:OG}
	Let $\{(x^k,y^k)\}$ be the sequence generated by Algorithm~\ref{alg:ASYSCA} under  Assumptions~\ref{ass:cost_functions}, \ref{ass:strong_conn}, \ref{ass:cost_functions_eb},~\ref{assump:surrogate}, and \ref{assumption_weights}.  Then, for $k \geq \bar{k}$, it holds\vspace{-0.2cm}
	\begin{equation}\label{eq:o}
		\begin{aligned}
			& \OG^{k+1}
			\leq  C_4(\gamma)  \, \zeta(\gamma)^{k} \OG^{0} + \frac{C_3(\gamma) C_4(\gamma)}{\zeta(\gamma)} \sum_{\ell = 0}^k  \zeta(\gamma)^{k-\ell} \FY^{\ell}.\vspace{-0.1cm}
		\end{aligned}
	\end{equation}
\end{proposition}

Since $\sigma(\gamma) < 1$ for $0<\gamma < \sup_{\epsilon \in (0, 2\tilde{\mu}) }\frac{2 \tilde{\mu} - \epsilon}{L} = \frac{2 \tilde{\mu}}{L}$ and $\eta \in (0,1]$,
Proposition~\ref{prop:OG} shows that, for sufficiently small  $\gamma>0$,   the optimality gap $\OG^k$ converges to zero R-linearly if $\FY^k$ does so.  The proof of Proposition~\ref{prop:OG} follows   from  Proposition \ref{prop:contraction}   and Lemma~\ref{lem:prod_mat_bound} below.  

\begin{proposition}\label{prop:contraction}
 	Let $\{(x^k,y^k)\}$ be the sequence generated by Algorithm~\ref{alg:ASYSCA} in the setting of Proposition~\ref{prop:OG}.
Let $p^k \triangleq \widetilde{U}(\augX^{k}) - U(x^*) \mathbf{1}$;   let $\D^k$ be  the diagonal matrix   with all diagonal entries $1$ and $\D_{i^k  i^k}^k = \sigma(\gamma)$; and let  $(\W\D)^{k:\ell} \triangleq \W^k\D^k \cdots \W^\ell\D^\ell$.
Then, for $k \geq \bar{k}$, \vspace{-0.2cm}
\begin{equation}\label{eq:og}
	\begin{aligned}
		& {p}^{k+1}  				\preccurlyeq  \left(  \W\D\right)^{k:0} {p}^{0} + C_3(\gamma) \sum_{\ell = 1}^k  \left(  \W \D  \right)^{k:\ell}  \W^{\ell-1} {e}_{i^{\ell-1}} \FY^{\ell-1}  + C_3(\gamma) \W^{k} {e}_{i^{k}} \FY^{k},
	\end{aligned}
\end{equation}
where    $C_3(\gamma)$ is defined in~\eqref{eq:C_3}.
\end{proposition}
\begin{proof}
By   convexity of $U$ and (\ref{eq:chi}), we have
\begin{equation}\label{eq:og}
	\begin{aligned}
		  {p}^{k+1}  &=  \widetilde{U}(\augX^{k+1}) -  U(x^*)  \mathbf{1}   \preccurlyeq  \W^k\, \left(\widetilde{U} \left(\augX^k + \gamma\Delta {\augX}^k \right)    - U(x^*)  \mathbf{1} \right).
	\end{aligned}
\end{equation}
Since $\widetilde{U} \left(\augX^k + \gamma\Delta {\augX}^k \right)$ differs from $\widetilde{U} \left(\augX^k \right)$ only by  its $i^k$-th row, we study   descent occurred at this row, which is   $(v_{i^k}^{k+1})^\top = (x_{i^k}^{k} + \gamma \left( \widetilde{x}_{i^k}^k - x_{i^k}^k\right))^\top.$
Recall that by applying the descent lemma on $F$ and using the convexity of $G$ we   proved
\begin{equation}\label{eq:ik_desc3}
	\begin{aligned}
		& U(v_{i^k}^{k+1}) -  U(x_{i^k}^{k})
		\leq \frac{L}{2} \gamma^2 \norm{\dx^k}^2  \\
		& ~~ + \gamma \underbrace{\left(  \nabla F(x_{i^k}^{k})^\top \left( \widetilde{x}_{i^k}^k - x_{i^k}^k\right)  +  G(\widetilde{x}_{i^k}^k) - G(x_{i^k}^{k})  \right)}_{T_1} .
	\end{aligned}
\end{equation}
 {The above inequality establishes a connections between    $U(v_{i^k}^{k+1})$ and $U(x_{i^k}^k)$.  However, it is not clear whether there is any contraction (up to some error) going from the optimality gap    $U(v_{i^k}^{k+1})-U^*$ to $U(x_{i^k}^k)-U^*$.
To investigate it,  we derive in the lemma below  two upper bounds of    $T_1$ in (\ref{eq:ik_desc3}), in terms of   $U(v_{i^k}^{k+1}) - U(x^*)$ and $\norm{\dx^k}$ (up to the tracking error).  Building on these bounds and (\ref{eq:ik_desc3}) we can finally prove  the desired contraction, as stated in \eqref{eq:U_cont}.}
\begin{lemma}\label{lem:bound_T}
	$T_1$ in (\ref{eq:ik_desc3}) can be bounded in the following two alternative ways: for $k \geq \bar{k}$,
	\begin{align}
		\label{eq:bound_T}T_1
		\leq & \left(- \tilde{\mu} + \frac{ \epsilon}{2} \right) \cdot \norm{\dx^k}^2 +  \frac{1}{2  \epsilon}\, \FY^k ,\\
		\label{eq:bound_T_1}T_1
		\leq &- \frac{1}{1 - \gamma} \left(U(v_{i^k}^{k+1}) - U(x^*)\right)  + \frac{1}{1-\gamma}\left(  c_5 \norm{\dx^k}^2 + c_6  \FY^k \right),
	\end{align}
	where $c_5$ and $c_6$ are polynomials in $(1,l,\tilde{l},L,\kappa)$ whose expressions are given in~\eqref{eq:const_expression}.
\end{lemma}

\begin{proof}
See Appendix~\ref{ap:lem:bound_T}.
\end{proof}


	Using  Lemma~\ref{lem:bound_T} in~\eqref{eq:ik_desc3}
	yields 
\begin{subequations}\label{eq:ik_desc4}
 	\begin{align}\label{eq:ik_desc4_1}
		\begin{split}
			&  U(v_{i^k}^{k+1}) - U^* \leq {}(1-\gamma) \left(U(x_{i^k}^{k}) - U^*) \right)  + \left(\frac{L}{2} \gamma(1-\gamma)  + c_5 \right) \gamma \norm{\dx^k}^2 + c_6 \cdot \gamma \FY^k \\
			& \leq {} (1-\gamma) \left(U(x_{i^k}^{k}) - U(x^*(x_{i^k}^{k})) \right)   +\underbrace{\left(c_5 + L/8\right)}_{c_7} \gamma \norm{\dx^k}^2 + c_6 \cdot \gamma \FY^k,\end{split}
			\end{align}  and \vspace{-.2cm}
		\begin{align}\label{eq:ik_desc4_2}	\begin{split} & U(v_{i^k}^{k+1}) - U^* \leq  U(x_{i^k}^{k}) - U^*    -  \left( \tilde{\mu} - \frac{\gamma L}{2} - \frac{ \epsilon }{2}\right) \gamma \norm{\dx^k}^2 + \frac{\gamma}{2 \epsilon}\FY^k.
		\end{split}
	\end{align}	\end{subequations}
	Canceling out $\norm{\dx^k}^2$ in (\ref{eq:ik_desc4_1})-(\ref{eq:ik_desc4_2}) yields:
for $k \geq \bar{k}$,
	\begin{align}\label{eq:U_cont}
		U(v_{i^k}^{k+1}) - U(x^*) \leq  \sigma( \gamma)  \left( U(x_{i^k}^{k}) - U(x^*) \right) + C_3(\gamma) \FY^k,
	\end{align}
	where   $\sigma(\gamma)$ and $C_3(\gamma)$ are defined in~\eqref{eq:C_3}.  Thus we observed a contraction from $\left( U(x_{i^k}^{k}) - U(x^*) \right)$ to $U(v_{i^k}^{k+1}) - U(x^*).$  Continuing from \eqref{eq:og}, we have
\begin{equation*}
	\begin{aligned}
		& {p}^{k+1}
		  \stackrel{\eqref{eq:U_cont}}{\preccurlyeq}  \W^k \left(\D^k {p}^{k} + C_3 (\gamma) \, \FY^k \,  {e}_{i^k} \right) \\
		&				\preccurlyeq  \left(  \W\D\right)^{k:0} {p}^{0} + C_3(\gamma) \sum_{\ell = 1}^k  \left(  \W \D  \right)^{k:\ell}  \W^{\ell-1} {e}_{i^{\ell-1}} \FY^{\ell-1}   + C_3(\gamma) \W^{k} {e}_{i^{k}} \FY^{k}.
	\end{aligned}\vspace{-0.4cm}
\end{equation*}
\end{proof}

The lemma below shows that the operator norm of   $ (\W \D)^{k:\ell}$ induced by the $\ell_\infty$ norm decays at a linear rate.

\begin{lemma} \label{lem:prod_mat_bound}
	For any $k\geq \ell  \geq 0,$
	\[
	\norm{(\W\D)^{k:\ell} }_\infty \leq C_4(\gamma)  \, \zeta(\gamma)^{k-\ell},
	\]
	where the expression of $\zeta(\gamma)$, $C_4(\gamma) $, and $K_1$ are given in~\eqref{eq:C_3}.
\end{lemma}

\begin{proof}
See Appendix~\ref{ap:pf_prod_mat_bound}.
\end{proof}

\subsubsection{Proof of  \eqref{eq:dx_bound}}\label{sec:opt_err}
Eq. \eqref{eq:dx_bound} follows directly from the second inequality of~\eqref{eq:ik_desc4} and the fact that $U(x_{i^k}^k) - U(v_{i^k}^{k+1})  \leq \OG^k$. This completes the proof of the inequality system~\eqref{eq:small_gain_inequalities}.\vspace{-0.2cm}

\subsection{Step 3: R-linear convergence via the generalized small gain theorem}\label{sec:geo_rate}
 The last step is to show that all the error quantities in (\ref{eq:small_gain_inequalities})   vanish at a linear rate. To do so, we leverage the  generalized small gain theorem \cite[Th. 17]{Tian_arxiv}. We use the following.

\begin{definition}[\cite{Nedich-geometric}]\label{df:sgt}
	Given the sequence $\{u^k \}_{k=0}^{\infty}$,   a constant  $\lambda \in (0,1)$, and   $N \in \mathbb{N}$, let us define
	$$\abs{u}^{\lambda,N} = \max_{k = 0,\ldots, N} \frac{\abs{u^k}}{\lambda^k}, \quad \abs{u}^{\lambda} = \sup_{k \in \mathbb{N}_0} \frac{\abs{u^k}}{\lambda^k}.$$
	If $\abs{u}^{\lambda} $ is upper bounded, then $u^k  = \mathcal{O} (\lambda^k)$, for all $k\in \mathbb{N}_0.$
\end{definition}

Invoking \cite[Lemma~20 \& Lemma~21]{Tian_arxiv}, if we choose $\lambda$ such that $\max \left( \rho^2, \zeta(\gamma) \right) < \lambda <1,$
by~\eqref{eq:small_gain_inequalities} we get
\begin{align}
	\label{ineq:1}&  \abs{\TE}^{\sqrt{\lambda} ,N}  \leq   \frac{3C_1 l}{\sqrt{\lambda} -\rho} (\abs{\CE}^{\sqrt{\lambda},N}  + \gamma \abs{\NT^k}^{\sqrt{\lambda},N} )   + \TE^0 + \frac{C_1\norm{{g}^0}}{\sqrt{\lambda}} \\
	\label{ineq:2}& \abs{\CE}^{\sqrt{\lambda},N}  \leq   \frac{C_2 \gamma}{\sqrt{\lambda}-\rho} \abs{\NT^k}^{\sqrt{\lambda},N}  + \CE^0 + \frac{C_2 \CE^0}{\sqrt{\lambda}} \\
	\label{ineq:3} & \abs{\OG} ^{\lambda,N} \leq \frac{C_3(\gamma) C_4(\gamma)}{\zeta(\gamma) \left( \lambda - \zeta(\gamma)\right)} \abs{\FY}^{\lambda,N} +  \OG^0 + \frac{C_4(\gamma)   \OG^{0} }{\lambda}\\
	\label{ineq:4} & \abs{\FY} ^{\lambda,N}   \leq  8 \, I\,l^2 \abs{(\CE)^2}^{\lambda,N} + 2 \, I^2  \abs{(\TE)^2}^{\lambda,N}   \\
	\label{ineq:5} & \abs{(\NT^k)^2}^{\lambda,N}  \leq  \frac{1}{2 \, \epsilon \left( \tilde{\mu} - \frac{ \epsilon}{2}- \frac{\gamma L}{2}  \right)  } \abs{\FY}^{\lambda,N}    +  \frac{1}{\gamma \left( \tilde{\mu} - \frac{ \epsilon}{2}- \frac{\gamma L}{2}  \right) }   \abs{\OG} ^{\lambda,N}
\end{align}
  Taking the square on both sides of~\eqref{ineq:1} \&~\eqref{ineq:2} while using $\left(\abs{u}^{q,N}\right)^2 = \abs{(u)^2}^{q^2,N}$,  and writing the result in matrix form we obtain \eqref{eq:smaill_gain} at the top of next page.
\begin{figure*}
\begin{equation} \label{eq:smaill_gain}
\left[ {\begin{array}{c}
	\abs{(\TE)^2}^{\lambda ,N} \\
	\abs{(\CE)^2}^{\lambda,N} \\
	\abs{\OG} ^{\lambda,N}\\
	\abs{\FY} ^{\lambda,N}   \\
	\abs{(\NT^k)^2}^{\lambda,N}
	\end{array} } \right]
\preccurlyeq
\underbrace{
	\left[{\begin{array}{ccccc}
		0                                      & \frac{36C_1^2 l^2}{(\sqrt{\lambda} -\rho)^2}      & 0 & 0                                                                                                                      & \frac{36C_1^2 l^2\gamma^2}{(\sqrt{\lambda} -\rho)^2} \\
		0                                       & 0                                                                   & 0 & 0                                                                                                                       &   \frac{3 C_2^2 \gamma^2}{(\sqrt{\lambda}-\rho)^2}  \\
		0                                      & 0                                                                  & 0 & \frac{C_3(\gamma) C_4(\gamma)}{\zeta(\gamma)  \left( \lambda - \zeta(\gamma)\right)} & 0  \\
		2I^2                                   &  8I\,l^2                                                      & 0  & 0                                                                                                                      & 0\\
		0                                        &  0                                                     &  \frac{1}{\gamma \left( \tilde{\mu} - \frac{  \epsilon}{2}- \frac{\gamma L}{2}  \right) }    &  \frac{1}{2 \,  \epsilon \,\left( \tilde{\mu} - \frac{ \epsilon}{2}- \frac{\gamma L}{2}  \right)  }             & 0
		\end{array} }
	\right]}_{\triangleq {G}}
\left[ {\begin{array}{c}
	\abs{(\TE)^2}^{\lambda ,N} \\
	\abs{(\CE)^2}^{\lambda,N} \\
	\abs{\OG} ^{\lambda,N}\\
	\abs{\FY} ^{\lambda,N}   \\
	\abs{(\NT^k)^2}^{\lambda,N}
	\end{array} } \right]
+ {\epsilon}^N.
\end{equation}
\vspace{-0.8cm}
\end{figure*}

We are now ready to apply  \cite[Th. 17]{Tian_arxiv}: 
 a sufficient condition for  $\TE,$ $\CE,$ $\OG,$ $\FY$, and $\NT^2$ to vanish  at an R-linear rate is $\rho({G})<1$. By \cite[Lemma~23]{Tian_arxiv},
 this is equivalent to requiring  $p_{G}(1)>0$, where  $p_{G}(z)$ is the characteristic polynomial  of ${G}$, This leads to the following condition:\vspace{-0.2cm}
\begin{align*}
\hspace{-0.3cm}	& \mathcal{B}(\lambda;\gamma) \\
	& =  \left( \frac{72\, I^2\,C_1^2 \, l^2 \,\gamma^2}{(\sqrt{\lambda} -\rho)^2}   + \frac{24 \,I \,l^2\, C_2^2 \gamma^2}{(\sqrt{\lambda}-\rho)^2} +  \frac{216 \,I^2\, C_1^2\,C_2^2\, l^2\,\gamma^2}{(\sqrt{\lambda} -\rho)^4}    \right)  \\
	&  \cdot \left( \frac{1}{2\epsilon \left( \tilde{\mu} - \frac{\epsilon}{2}- \frac{\gamma L}{2}  \right) }      +  \frac{C_3(\gamma) C_4(\gamma)}{\zeta(\gamma)  \left( \lambda - \zeta(\gamma)\right)}  \frac{1}{\gamma \left( \tilde{\mu} - \frac{\epsilon}{2}- \frac{\gamma L}{2}  \right)  }  \right)  < 1.
\end{align*}

It is not hard to see that  $\mathcal{B}(\lambda; \gamma)$ is continuous at $\lambda =1$, for any $\gamma \in (0, \frac{2 \tilde{\mu} -\epsilon}{L})$. Therefore, as long as
\begin{equation}\label{eq:b_1_gamma}
	\begin{aligned}
		&  \mathcal{B}(1;\gamma) = \left(\frac{72\, I^2\,C_1^2 \, l^2 }{(1 -\rho)^2}   + \frac{24 \,I \,l^2\, C_2^2}{(1-\rho)^2} +  \frac{216 \,I^2\, C_1^2\,C_2^2\, l^2}{(1 -\rho)^4}       \right) \gamma \cdot \\
		&  \left( \frac{\gamma}{2\epsilon\left( \tilde{\mu} - \frac{\epsilon}{2}- \frac{\gamma L}{2}  \right)   }      +  \frac{C_3(\gamma) C_4(\gamma)}{\zeta(\gamma)  \left( 1 - \zeta(\gamma)\right)}  \frac{1}{\left( \tilde{\mu} - \frac{\epsilon}{2}- \frac{\gamma L}{2}  \right) }  \right) < 1,
	\end{aligned}
\end{equation}
there will exist some $\lambda \in (0,1)$ such that $\mathcal{B}(\lambda;\gamma)<1.$

We show now that $ \mathcal{B}(1;\gamma)  < 1$, for sufficiently small $\gamma$. We only need to prove boundedness of the following quantity when $\gamma\downarrow 0$:
$$ \frac{C_3(\gamma) C_4(\gamma)}{\zeta(\gamma) \left( 1- \zeta(\gamma)\right)}  = \underbrace{\frac{  c_6  (\tilde{\mu} - \frac{ \epsilon}{2}- \frac{\gamma L}{2} ) +  \frac{c_7}{2 \epsilon}   }{ \left(  c_7 +  \tilde{\mu} - \frac{ \epsilon}{2}- \frac{\gamma L}{2}   \right)   \zeta(\gamma)^{K_1+1}}}_{\triangleq h(\gamma)} \cdot \frac{\gamma}{1-\zeta(\gamma)}.$$
It is clear that $h(\gamma)$ is right-continuous at $0 $ and thus $\lim_{\gamma \downarrow 0} h(\gamma)<\infty$.  Hence,  it is left to check that $\frac{\gamma}{1-\zeta(\gamma)}$ is bounded when  $\gamma \downarrow 0.$
According to L'H$\hat{\text{o}}$pital's rule,
\begin{align*}
	& \lim_{\gamma\, \downarrow 0}  \frac{\gamma}{1-\zeta(\gamma)}
	= -\frac{K_1}{\left( 1 - \left(1- \sigma(\gamma) \right) \eta \right)^{\frac{1}{K_1}-1}} \frac{1}{\eta \sigma'(\gamma)}\bigg\vert_{\gamma = 0} =  \frac{K_1 \left( c_7+  \tilde{\mu}- \frac{\epsilon}{2}  \right)}{ \eta \left( \tilde{\mu}- \frac{\epsilon}{2}  \right) } < \infty.
\end{align*}

Finally, we prove  that all $(x_i^k)_{k \geq \bar{k}}$   converge  linearly to some $x^\star$. By the definition of the augmented matrix $\augX$ and the update~\eqref{eq:chi}, we have: for $k \geq \bar{k}$,
\begin{align*}
	& \| h^{k+1} - h^{k}\|  = \| (\W - I) h^k + \gamma \augY^k\|\\
	& \leq \| (\W - I) (\augX^{k}- \bm{1} \cdot (\CV^{k})^\top)\| + \gamma \|\augY^k\| \leq 3 \CE^k + \gamma \NT^k.
\end{align*}
Since both $\CE^k$ and $\NT^k$ are $\mathcal{O}\left( (\sqrt{\lambda})^k \right)$, $\sum_{k=0}^{\infty}\|h^{k+1} - h^{k} \| < + \infty$;   thus $\{\augX^k\}_{k \in \mathbb{N}}$ is Cauchy and converges to some $\mathbf{1} (x^\star )^\top$, implying all $x_i^k$ converges to $x^\star$.  We prove next that $x_i^k$ converges to $x^\star$ R-linearly. For any $k' > k \geq \bar{k}$, we have $\|\augX^k - \augX^{k'}\| \leq \sum_{t = k}^{k' - 1} \| \augX^t - \augX^{t+1}\|  \leq \sum_{t = k}^{k' - 1} \left(3 \CE^t + \gamma \NT^t\right) = \mathcal{O} \left( (\sqrt{\lambda})^k \right)$.
Taking $k' \to \infty$ completes the proof.


\section{Proof of Theorem~\ref{thm:sublinear_const}}\label{sec:pf_2}
In this section we prove the sublinear convergence of \ASYSCA/. We organize the proof in two steps. Step 1: we prove     $\sum_{k=0}^{\infty} (\NT^k)^2 < +\infty$ by showing the descent of a properly constructed Lyapunov function. This function represents a major novelty of our analysis--see Remark \ref{rmk:Lyapunov}. Step 2:  we connect the decay rate of $\NT^k$  and that of the merit function $M_F(x^k)$.\vspace{-0.2cm}

\subsection{Step 1:  $\NT^k$ is square summable}
In   Sec.~\ref{sec:prelim} we have shown that the weighted average of the local variables $\CV$ evolves according to the   dynamics   Eq.~\eqref{eq:avg_dynamic_x}.
Using  $\CV^0 = {\bd{{\psi}}^0}^{\top}{\augX}^0$,  \eqref{eq:avg_dynamic_x} 
can be rewritten recursively as
\begin{align}
	\CV^{k+1} = \CV^k + \gamma {\psi^k}^\top \Delta \augX^k = \CV^k + \gamma \psi_{i^k}^k \dx^k.
\end{align}

Invoking the descent lemma 
while recalling  $\NT^k = \norm{\dx^k}$,
yields \vspace{-0.2cm}
\begin{equation}\label{descent}
	\begin{aligned}
		&  F(\CV^{k+1}) \leq   F(\CV^{k}) + \gamma \psi_{i^k}^k \nabla F(\CV^{k})^\top \dx^k + \frac{L(\gamma \psi_{i^k}^k)^2}{2} (\NT^k)^2 \\%
		& \overset{\eqref{eq:mini_prin}}{\leq } 
		F(\CV^{k}) + \frac{L \gamma^2}{2} (\NT^k)^2 - \gamma\psi_{i^k}^k \left(\tilde{\mu}  (\NT^k)^2+G(\tl{x}_{i^k}^k)  - G(x_{i^k}^k) \right) \\
		&\qquad  + \gamma \psi_{i^k}^k \left(\nabla F(\CV^{k})-I\avg{g}^k \right)^\top \dx^k    + \gamma\psi_{i^k}^k  \left(I\avg{g}^k-Iy_{i^k}^k \right)^\top \dx^k \\%
		& \leq  F(\CV^{k}) + \frac{L \gamma^2}{2} (\NT^k)^2 - \gamma\psi_{i^k}^k \left(\tilde{\mu}  (\NT^k)^2+G(\tl{x}_{i^k}^k)  - G(x_{i^k}^k) \right)   + \gamma l \sqrt{I} \CE^k \, \NT^k + \gamma  I \TE^k \, \NT^k.
	\end{aligned}
\end{equation}

Introduce the Lyapunov function
\begin{align}\label{eq:noncon_lyapu}
	L^k \triangleq  F(\CV^{k}) +{{\psi}^k}^{\top}\widetilde{G}(\augX^k)
\end{align}
where  $\widetilde{G}: \mathbb{R}^{S \times n} \to \mathbb{R}^S$ is defined as $\widetilde{G}(\augX) \triangleq [G(\augx_1),\cdots,G(\augx_S)]^{\top}$, for $\augX = [\augx_1,\cdots, \augx_S]^\top \in \mathbb{R}^{S \times n}$.
\begin{remark}\label{rmk:Lyapunov}
Note that $L^k$ contrasts with the functions used in the literature of distributed algorithms to study convergence in the nonconvex setting. Existing choices either cannot deal with asynchrony \cite{YingMAPR,sun2019convergence} (e.g. the unbalance in the update frequency of the agents and the use of outdated information)   or cannot handle nonsmooth functions in the objective and constraints \cite{Tian_arxiv}.  A key feature of $L^k$ is to  combine current and past information throughout suitable dynamics, $\{\CV^k\}$, and weights  averaging via $\{{\psi}^k\}$. \end{remark}

Using the dynamics of $\augX^k$ as in \eqref{eq:chi}, 
we get
$$	\widetilde{G}(\augX^{k+1}) \preccurlyeq  \W^k \left((1-\gamma)\widetilde{G}({\augX}^k)+\gamma \widetilde{G}(\augX^k +\augY^k) \right).$$
where we used the convexity of $G$ and the row-stochasticity of $\W^k$.
Thus
\begin{align*}
	& {{\psi}^{k+1}}^{\top}\widetilde{G}(\augX^{k+1}) \\
	& \leq {{\psi}^{k+1}}^{\top}\W^k \left((1-\gamma)\widetilde{G}({\augX}^k)+\gamma \widetilde{G}(\augX^k +\augY^k) \right)\\
	& = {{\psi}^k}^{\top} \left((1-\gamma)\widetilde{G}({\augX}^k)+\gamma \widetilde{G}(\augX^k +\augY^k) \right),
\end{align*}
where in the last equality we used ${{\psi}^{t+1}}^{\top} \W\spe{t} = {{\psi}^t}^{\top}$ [cf.~\eqref{eq:prod_mat}].
Therefore,
\begin{align*}
	& \gamma \psi_{i^k}^k \left(G(x_{i^k}^{k})- G(\tl{x}_{i^k}^{k}) \right) \\
	& = \gamma \left({{\psi}^k}^{\top} \widetilde{G}(\augX^k) - {{\psi}^k}^{\top} \widetilde{G}(\augX^k +\augY^k) \right) \\
	& \leq  {{\psi}^k}^{\top}\widetilde{G}(\augX^k)-{{\psi}^{k+1}}^{\top}\widetilde{G}(\augX^{k+1}).
\end{align*}

Combining the above inequality with (\ref{descent}), we get
\begin{equation}\label{lyapunov_eq}
	\begin{aligned}
		& L^{k+1}
		\leq  {} L^k  -  \eta \tilde{\mu}  (\NT^k)^2  \gamma +\frac{L}{2} (\NT^k)^2 \gamma^2 + \frac{\epsilon_1}{2}  l^2 I (\CE^k)^2   +  \frac{1}{2 \epsilon_1} \gamma^2 (\NT^k)^2   +  \frac{\epsilon_2}{2}I^2 (\TE^k)^2 + \frac{1}{2\epsilon_2} \gamma^2 (\NT^k)^2  \\
		& =   L^k   -   (\NT^k)^2  \gamma \left( \eta \tilde{\mu} - \gamma \left( \frac{L}{2} +  \frac{1}{2 \epsilon_1} +\frac{1}{2 \epsilon_2} \right)\right)   + \frac{\epsilon_1}{2}  l^2 I (\CE^k)^2 +    \frac{\epsilon_2}{2}I^2 (\TE^k)^2 \\
		& \leq {} L^0  - \sum_{t=0}^k (\NT^t)^2  \gamma \left( \eta \tilde{\mu} - \gamma \left( \frac{L}{2} +  \frac{1}{2 \epsilon_1} +\frac{1}{2 \epsilon_2} \right)\right)    + \frac{\epsilon_1}{2}  l^2 I  \sum_{t=0}^k {\CE^t}^2 + \frac{\epsilon_2}{2}I^2  \sum_{t=0}^k {\TE^t}^2 .
	\end{aligned}
\end{equation}
To bound the last two terms in \eqref{lyapunov_eq}, we apply Proposition~\ref{lm:sqr_sum}:\vspace{-0.2cm}
\begin{align*}\label{eq:lyapunov}
	\begin{split}
		& \quad L^{k+1} \leq L^0 - \sum_{t=0}^k (\NT^t)^2  \gamma \bigg( \eta \tilde{\mu} - \gamma \bigg( \frac{L}{2} +  \frac{1}{2 \epsilon_1} +\frac{1}{2 \epsilon_2}   + \frac{\epsilon_1}{2}  l^2 I \varrho_{x}  + \frac{\epsilon_2}{2}I^2 \varrho_{y}  \bigg)\bigg) + \frac{\epsilon_1}{2}  l^2 I  c_{x} + \frac{\epsilon_2}{2}I^2  c_{y}\\
		& = L^0 - \sum_{t=0}^k (\NT^t)^2  \gamma \left( \eta \tilde{\mu} - \gamma \left( \frac{L}{2}  + \sqrt{l^2 I \varrho_{x}} +\sqrt{I^2 \varrho_{y}} \right)\right)  + \frac{ l^2 I  c_{x}}{2\sqrt{l^2 I \varrho_{x}}}  + \frac{I^2  c_{y}}{2\sqrt{I^2 \varrho_{y}}},
	\end{split}
\end{align*}
where in the last equality we set $\epsilon_1 = 1/\sqrt{l^2 I \varrho_{x}}$ and $\epsilon_2 =1/\sqrt{I^2 \varrho_{y}}$.
Note that
\begin{align*}
& L^k = F(\CV^{k}) +{{\psi}^k}^{\top}\widetilde{G}(\augX^k) \\
& \geq F(\CV^{k}) + G({{\psi}^k}^{\top}\augX^k) = U(\CV^{k}) \geq U^\ast,\end{align*}
for all $k \in \mathbb{N}_+$. Thus, for sufficiently small  $\gamma$, such that
\begin{equation}\label{eq:noncon_step}
	\gamma  \leq \bar{\gamma}_{ncvx}\triangleq  \eta \tilde{\mu} \left( L  + 2\sqrt{l^2 I \varrho_{x}} + 2\sqrt{I^2 \varrho_{y}} \right)^{-1},
\end{equation}
we can obtain the following bound
\begin{align}\label{eq:61}
	\sum_{t=0}^\infty (\NT^t)^2 \leq \frac{ 2L^0 - 2U^*+  \frac{ l^2 I  c_{x}}{\sqrt{l^2 I \varrho_{x}}}  + \frac{I^2  c_{y}}{\sqrt{I^2 \varrho_{y}}}}{\gamma \eta \tilde{\mu}}.
\end{align}

\subsection{Step 2: $M_F (x^k)$ vanishes at sublinear rate}
In this section we establish the connection between $M_F (x^k)$ and $\NT^k$, $\CE^k$, and $\TE^k$.

Invoking Lemma~\ref{lm:resi} we can bound $\|\bar{x}^k - \prox_G (\bar{x}^k - \nabla F (\bar{x}^k))\|$ as\vspace{-0.2cm}
\[
\begin{aligned}
& \|\bar{x}^k - \prox_G (\bar{x}^k - \nabla F (\bar{x}^k))\|^2 \\
\leq {}& 3 \|\bar{x}^k - x_{i^k}^k\|^2 +3 \|x_{i^k}^k - \prox_G(x_{i^k}^k - \nabla F (x_{i^k}^k))\|^2  + 3 \|\prox_G(x_{i^k}^k - \nabla F (x_{i^k}^k)) - \prox_G(\bar{x}^k - \nabla F (\bar{x}^k))\|^2\\
\overset{(*)}{ \leq } & 3 \|\bar{x}^k - x_{i^k}^k\|^2 +3 \|x_{i^k}^k - \prox_G(x_{i^k}^k - \nabla F (x_{i^k}^k))\|^2  + \|x_{i^k}^k - \nabla F (x_{i^k}^k) - (\bar{x}^k - \nabla F (\bar{x}^k))\|^2\\
\leq  & (5 + 2 L^2) \|\bar{x}^k - x_{i^k}^k\|^2 +3 \|x_{i^k}^k - \prox_G(x_{i^k}^k - \nabla F (x_{i^k}^k))\|^2 \\
\leq & 4 (5 + 2 L^2) (\CE^k)^2 +3 \|x_{i^k}^k - \prox_G(x_{i^k}^k - \nabla F (x_{i^k}^k))\|^2 \\
\leq & 4 (5 + 2 L^2) (\CE^k)^2 +3\left(  4 \left( 1 + (l + \tilde{l})^2 \right) (\NT^k)^2 + 5 \FY^k \right),
\end{aligned}
\]
where  (*) follows from the nonexpansiveness of a proximal operator. Further applying Lemma~\ref{lm:sqr_sum} and~\eqref{eq:tracking_bound}, yields:
\begin{align*}
	\begin{split}
		&\sum_{t=0}^{k}M_F (x^t) \\
		\leq & \sum_{t=0}^k  \|\bar{x}^k - \prox_G (\bar{x}^k - \nabla F (\bar{x}^k))\|^2 + \sum_{t=0}^k (\CE^t)^2 \\
		\leq & \sum_{t=0}^k \left( (21 + 8L^2) (\CE^t)^2 +  3 \left( 4 (1+ (l + \tilde{l})^2) (\NT^t)^2 + 5 \FY^t \right) \right)\\
		\leq &  \sum_{t=0}^k \left( (21 + 8L^2)(\CE^t)^2 + 15 \left( 8I l^2 (\CE^t)^2 + 2 I^2 (\TE^t)^2\right) \right) 	 + 12 (1+ (l + \tilde{l})^2) \sum_{t=0}^k (\NT^t)^2\\
		\leq & \left(21 + 8L^2 + 120  I l^2\right) \left( c_{x} + \varrho_{x} \sum_{t=0}^{k} \gamma^2 (\NT^t)^2 \right)   + 30  I^2 \left(  c_{y} + \varrho_{y} \sum_{t=0}^{k} \gamma^2 (\NT^t)^2  \right) + 12 (1+ (l + \tilde{l})^2) \sum_{t=0}^k (\NT^t)^2\\
		= & \big( \left(21 + 8L^2 + 120 I l^2\right) \varrho_{x} \gamma^2+ 30 \kappa^2 I^2 \varrho_{y} \gamma^2   	+ 12  (1+ (l + \tilde{l})^2) \big) \\
		&  \sum_{t=0}^k (\NT^t)^2  + \left(21 + 8L^2 + 120  I l^2\right) c_x  + 30 \kappa^2 I^2 c_y\\
		\overset{\eqref{eq:61}}{\leq }& \big(  \left(21 + 8L^2 + 120 I l^2\right) \varrho_{x} \gamma^2+ 30 \kappa^2 I^2 \varrho_{y} \gamma^2   + 12  (1+ (l + \tilde{l})^2) \big) \left(\frac{ 2L^0 - 2U^* +  \frac{ l^2 I  c_{x}}{\sqrt{l^2 I \varrho_{x}}}  + \frac{I^2  c_{y}}{\sqrt{I^2 \varrho_{y}}}}{ \gamma \eta \tilde{\mu}}\right) \\
		& + \left(21 + 8L^2 + 120  I l^2\right) c_x + 30 \kappa^2 I^2 c_y \triangleq B_{opt},
	\end{split}
\end{align*}
where     $\varrho_{x}$ and $\varrho_{y}$ are defined in Lemma~\ref{lm:sqr_sum}.

Let $T_\delta = \inf \{ k \in \mathbb{N} \,|\, M_F(x^k) \leq \delta\}$. Then it holds:
$
T_\delta \cdot \delta < \sum_{k=0}^{T_\delta - 1}  M_F(x^k) \leq B_{opt}
$ and thus $T_\delta = \mathcal{O} (B_{opt} /\delta)$.
 	\section{Conclusion}\vspace{-0.1cm}
	We proposed \ASYSCA/, an asynchronous decentralized method for multiagent convex/nonconvex composite minimization problems over (di)graphs. The algorithm employs SCA techniques and is robust against agents' uncoordinated activations and use of outdated information (subject to arbitrary  but bounded delays). For convex (not strongly convex) objectives satisfying the LT error bound condition,   \ASYSCA/ achieves R-linear convergence rate while
	 sublinear convergence is   established  for nonconvex objectives.\vspace{-0.3cm}

\appendix\vspace{-0.1cm}
\section{Proof of Lemma~\ref{lm:tracking_err}}\label{ap:lm:tracking_err}
	Applying \cite[Th.~6]{Tian_arxiv} with the identifications:
	$	\epsilon^t = \nabla f_{i^t}(x_{i^t}^{t+1}) - \nabla f_{i^t}(x_{i^t}^{t}) $
	and
	\begin{align*}
& \mathfrak{m}^{k}_z = \sum_{i=1}^I {{z}_i^0} + \sum_{t=0}^{k-1} \epsilon^t   = \sum_{i=1}^I \nabla f_i (x_i^0) + \sum_{t=0}^{k-1} \left( \nabla f_{i^t}(x_{i^t}^{t+1}) - \nabla f_{i^t}(x_{i^t}^{t}) \right) \\
& \overset{(*)}{=}  I \cdot  \underbrace{\frac{1}{I}\sum_{i=1}^I \nabla f_i (x_i^k)}_{\triangleq \bar{g}^k},
	\end{align*}
	we arrive at
	$\TE^{k+1}  \leq   C_1 \left( \rho^k \norm{{g}\spe{0}} + \sum_{l = 0}^k \rho^{k-l} \|\epsilon^l\|\right),
	$
	where in $(*)$ we have used $x_j^{t+1} = x_j^t$ for $j\neq i^t.$
	The rest of the proof follows the same argument as in~\cite[Prop.~18]{Tian_arxiv}.\hfill $\square$

\section{Proof of Lemma~\ref{lem:bound_T}}\label{ap:lem:bound_T}
  	  Using  \eqref{eq:mini_prin}, we have: for any $\epsilon > 0$,
	\begin{align*}
		\begin{split}
			T_1 = & \left( \nabla F(x_{i^k}^{k}) \pm  Iy_{i^k}^k\right)^\top \left( \widetilde{x}_{i^k}^k - x_{i^k}^k\right)  +  G(\widetilde{x}_{i^k}^k) - G(x_{i^k}^{k}) \\
			\leq & - \tilde{\mu} \cdot \norm{\dx^k}^2 + \frac{1}{2  \epsilon}\, \FY^k +  \frac{ \epsilon}{2}\, \norm{\dx^k}^2.
		\end{split}
	\end{align*}
	Next we prove~\eqref{eq:bound_T_1}. For any $z \in \mathcal{K}$, let $x^*(z) \in \mathcal{P}_{\mathcal{K}^*} (z)$.
	By the Mean Value Theorem, there exists $\xi^k = \beta \, x^*(x_{i^k}^{k}) + (1-\beta) v_{i^k}^{k+1}$, with  $\beta \in (0,1)$, such that
	\begin{align}\label{eq:MVT}
		\begin{split}
		&	U(v_{i^k}^{k+1}) - U(x^*(x_{i^k}^{k}))
			=   \nabla F(\xi^k)^\top \big(v_{i^k}^{k+1} - x^*(x_{i^k}^{k}))   + G(v_{i^k}^{k+1} \big) - G(x^*(x_{i^k}^{k}) ).
		\end{split}
	\end{align}
	To deal with the inner product term, we invoke the algorithmic update~\eqref{eq:sca} and the first order optimality condtion~\eqref{eq:FOC_general} (with $z = x^*(x_{i^k}^{k}) $):
	\begin{align*}
		\begin{split}
			& \left( \nabla \widetilde{f}_{i^k}(\tl{x}_{i^k}^k;x_{i^k}^k)+Iy_{i^k}^k-\nabla f_{i^k}(x_{i^k}^k) \right)^\top \left( v_{i^k}^{k+1} - x^*(x_{i^k}^{k}) \right)  \\
			= & \left( \nabla \widetilde{f}_{i^k}(\tl{x}_{i^k}^k;x_{i^k}^k)+Iy_{i^k}^k-\nabla f_{i^k}(x_{i^k}^k) \right)^\top \big(\tl{x}_{i^k}^k - x^*(x_{i^k}^{k})   + (\gamma-1) (\tl{x}_{i^k}^k - x_{i^k}^k) \big)  \\
			\leq & - (1 - \gamma) \left( \nabla \widetilde{f}_{i^k}(\tl{x}_{i^k}^k;x_{i^k}^k)+Iy_{i^k}^k-\nabla f_{i^k}(x_{i^k}^k) \right)^\top  (\tl{x}_{i^k}^k - x_{i^k}^k) + G(x^*(x_{i^k}^{k}) ) - G(\tl{x}_{i^k}^k).
		\end{split}
	\end{align*}
	Therefore
	\begin{equation}\label{eq:ik_desc1} 		\begin{aligned}
			& U(v_{i^k}^{k+1}) - U(x^*(x_{i^k}^{k})) \\
			= & \left( \nabla F(\xi^k)   \pm \big(\nabla \widetilde{f}_{i^k}(\tl{x}_{i^k}^k;x_{i^k}^k)+Iy_{i^k}^k-\nabla f_{i^k}(x_{i^k}^k) \big) \right)^\top   \big(v_{i^k}^{k+1} - x^*(x_{i^k}^{k})) + G(v_{i^k}^{k+1} \big) - G(x^*(x_{i^k}^{k}) )  \\
			\leq &  \left( \nabla \widetilde{f}_{i^k}(\tl{x}_{i^k}^k;x_{i^k}^k)+Iy_{i^k}^k-\nabla f_{i^k}(x_{i^k}^k) \right) (v_{i^k}^{k+1} - x^*(x_{i^k}^{k}))   +  G(v_{i^k}^{k+1} ) - G(x^*(x_{i^k}^{k}) ) \\
			&+
			\underbrace{
			 \bigg( \norm{\nabla F(\xi^k) - \nabla F(x_{i^k}^k)} + \norm{\nabla F(x_{i^k}^k) -Iy_{i^k}^k } 
			 + \norm{\nabla \widetilde{f}_{i^k}(\tl{x}_{i^k}^k;x_{i^k}^k) -\nabla f_{i^k}(x_{i^k}^k) } \bigg)  \norm{v_{i^k}^{k+1} - x^*(x_{i^k}^{k})}
			}_{R_1}  \\
			\leq &- (1 - \gamma) \left( \nabla \widetilde{f}_{i^k}(\tl{x}_{i^k}^k;x_{i^k}^k)+Iy_{i^k}^k-\nabla f_{i^k}(x_{i^k}^k) \right)^\top   (\tl{x}_{i^k}^k - x_{i^k}^k)
			-  (1-\gamma) (G(\tl{x}_{i^k}^k) - G(x_{i^k}^k) )
			+ R_1,
		\end{aligned}
	\end{equation}
	where in the last inequality we have used the convexity of $G$.
	We thus arrive at the following bound on $T_1$:
	\begin{equation}\label{eq:inner_prod_bound}
		\begin{aligned}
			T_1  =  & \left( \nabla F(x_{i^k}^{k}) \pm \left( \nabla \widetilde{f}_{i^k}(\tl{x}_{i^k}^k;x_{i^k}^k)+Iy_{i^k}^k-\nabla f_{i^k}(x_{i^k}^k) \right)\right)^\top   \left( \widetilde{x}_{i^k}^k - x_{i^k}^k\right)  +  G(\widetilde{x}_{i^k}^k) - G(x_{i^k}^{k}) \\
			\leq & - \frac{1}{1 - \gamma} \left(U(v_{i^k}^{k+1}) - U(x^*(x_{i^k}^{k}))\right) +\frac{1}{1 - \gamma}  \cdot R_1 \\
			 + &\underbrace{
			 \bigg( \norm{\nabla \widetilde{f}_{i^k}(\tl{x}_{i^k}^k;x_{i^k}^k) -\nabla f_{i^k}(x_{i^k}^k)  } 
			+  \norm{  \nabla F(x_{i^k}^{k})  - Iy_{i^k}^k }   \bigg) \norm{ \dx^k} }_{R_2}.
		\end{aligned}
	\end{equation}
	It remains to bound the remainder terms $R_{1}$ and $R_2$.
	Note that
	\begin{align*}
		& \norm{v_{i^k}^{k+1} - x^*(x_{i^k}^{k})} = \norm{v_{i^k}^{k+1} \pm x_{i^k}^k - x^*(x_{i^k}^{k})}  \leq \text{dist} ( x_{i^k}^k, \mathcal{K}^*) + \gamma \,\norm{\dx^k}, \\
		& \norm{\xi^k - x_{i^k}^k}  \leq \beta \norm{x_{i^k}^{k} - x^*(x_{i^k}^{k})} + (1-\beta)\norm{v_{i^k}^{k+1} - x_{i^k}^k}   \leq \text{dist} ( x_{i^k}^k, \mathcal{K}^*) + \gamma \, \norm{\dx^k}.
	\end{align*}
	Applying Lemma~\ref{lm:resi}  and Corollary~\ref{cor:err_bound}, the following holds: for $k \geq \bar{k}$,
	\begin{align*}
		\left( \text{dist} ( x_{i^k}^k, \mathcal{K}^*) \right)^2 \leq  \kappa^2 \norm{x_{i^k}^k - \text{prox}_G (x_{i^k}^k -\nabla F(x_{i^k}^k))}^2  \leq \kappa^2 \left( 4 \left( 1 + (l + \tilde{l})^2 \right) \norm{\dx^k}^2 + 5 \FY^k \right).
	\end{align*}
	With the above inequalities and using the fact that $\gamma \leq 1$ we can bound $R_1$ as\vspace{-0.2cm}
	\begin{align*}
		\begin{split}
		\!\!\!\!\!	R_1 \leq & \norm{\nabla F(\xi^k) - \nabla F(x_{i^k}^k)}^2 +  \norm{\nabla F(x_{i^k}^k) -Iy_{i^k}^k }^2  +  \norm{\nabla \widetilde{f}_{i^k}(\tl{x}_{i^k}^k;x_{i^k}^k) -\nabla f_{i^k}(x_{i^k}^k) }^2  +  \norm{v_{i^k}^{k+1} - x^*(x_{i^k}^{k})}^2\\
			\leq &  L^2 \norm{\xi^k- x_{i^k}^k}^2 +  \FY^k +  (l + \widetilde{l})^2 \norm{\dx^k}^2  + 2 \, \text{dist} ( x_{i^k}^k, \mathcal{K}^*)^2 + 2  \gamma^2 \,\norm{\dx^k}^2\\
			\leq & \left( 2 L^2 + 2 \right)\text{dist} ( x_{i^k}^k, \mathcal{K}^*)^2 +  \FY^k   + \left(  2 L^2 \gamma^2 + 2  \gamma^2 +  (l + \widetilde{l})^2\right) \norm{\dx^k}^2 \\
			\leq & \underbrace{\left( 8  \kappa^2 (L^2 + 1) \big( 1 + (l + \tilde{l})^2 \big) + 2 L^2  + 2   +  (l + \widetilde{l})^2 \right)}_{c_3} \norm{\dx^k}^2    + \underbrace{\left( 10 \kappa^2 (L^2 + 1) + 1 \right)}_{c_4}
			\FY^k.
		\end{split}
	\end{align*}
	Similarly, $R_2$ can be bounded as
	\begin{align*}
		\begin{split}
			& R_2 \leq \norm{\nabla F(x_{i^k}^k) -Iy_{i^k}^k }^2 +  \norm{\nabla \widetilde{f}_{i^k}(\tl{x}_{i^k}^k;x_{i^k}^k) -\nabla f_{i^k}(x_{i^k}^k) }^2   + \norm{\dx^k}^2 \leq  \FY^k + \left( 1+ (l + \widetilde{l})^2 \right)\norm{\dx^k}^2.
		\end{split}
	\end{align*}
	Substituting the bounds of $R_1$ and $R_2$ in  \eqref{eq:inner_prod_bound} yields
	\begin{align*}
		\begin{split}
			& T_1 \leq  - \frac{1}{1 - \gamma} \left(U(v_{i^k}^{k+1}) - U(x^*(x_{i^k}^{k}))\right) +\FY^k  +\frac{1}{1 - \gamma}  \cdot \left( c_3 \norm{\dx^k}^2 + c_4  \FY^k \right)  + \left( 1+ (l + \widetilde{l})^2 \right)\norm{\dx^k}^2\\
			& \leq  - \frac{1}{1 - \gamma} \left(U(v_{i^k}^{k+1}) - U(x^*(x_{i^k}^{k}))\right)   + \frac{1}{1-\gamma} \left( c_5 \norm{\dx^k}^2 + c_6  \FY^k \right),
		\end{split}
	\end{align*}
	where \vspace{-0.3cm}
	\begin{align}\label{eq:const_expression}
		\hspace{-0.3cm}\begin{split}
			c_5  & \triangleq  8  \kappa^2 (L^2 + 1) \big( 1 + (l + \tilde{l})^2 \big)  + 2 L^2   + 2   +  (l + \widetilde{l})^2 + 1+ (l + \widetilde{l})^2,\\
			c_6 & \triangleq 10 \kappa^2 (L^2 + 1) + 2. 
		\end{split}\vspace{-1cm}
	\end{align}
\hfill $\square$
\section{Proof of Lemma~\ref{lem:prod_mat_bound}}\label{ap:pf_prod_mat_bound}
We know from Lemma~\ref{lm:Wscramb}~(ii): for all $k \geq 0$, all elements in the first $I$ columns of $\W^{k + K_1 -1:k}$ are no less than $\eta$.
	Since $\W^{k+K_1-1:k}\D^k$ is nonnegative, we have for each $i\in[S]$
	\begin{align*}
	& \norm{\W^{k+K_1-1:k}\D^k }_\infty \leq \max_{i=1,\ldots,S} \left\{  1 -  \big(1- \sigma(\gamma) \big) \W^{k+K_1-1:k}_{i,i^k}  \right\} \\
	& \leq  1 - \big(1- \sigma(\gamma) \big) \eta .
	\end{align*}

	On the other hand, because ${0} \preccurlyeq \D^k \preccurlyeq {I}\,$ for all $k$, we know
	$\label{eq:w_non_in}   \big(\W \D \big)^{m:k} \preccurlyeq \W^{m:k} \,\D^k, \,\,  \forall m\geq k. $
	Thus
	\begin{equation*}
		\norm{  (\W \D)^{k+K_1-1:k}  }_\infty \leq  \norm{\W^{k+K_1-1:k} \D^k}_\infty \leq 1 - \left(1- \sigma(\gamma) \right) \eta.
	\end{equation*}
	Finally for any $k\geq \ell  \geq 0,$
	\[
	\begin{aligned}
	& \norm{ (\W \D)^{k:\ell}  }_\infty \leq  \Bigg( \prod_{t = 1}^{\floor{\frac{k+1-\ell}{K_1}}}  \norm{\left(\W \D  \right)^{\ell+t\, K_1 -1:\ell+(t-1)K_1}}_\infty \Bigg) \\
	& ~~~~ \norm{\left(\W \D  \right)^{k:\ell+\floor{\frac{k+1-\ell}{K_1}}K_1} }_\infty\\
	&  \leq  \prod_{t = 1}^{\floor{\frac{k+1-\ell}{K_1}}}  \norm{\left(\W \D  \right)^{\ell+t\, K_1 -1:\ell+(t-1)K_1}}_\infty    \\
	& \leq  \left( 1 - \big(1- \sigma(\gamma) \big) \eta \right)^{\floor{\frac{k+1-\ell}{K_1}}} \leq \left( 1 - \big(1- \sigma(\gamma) \big) \eta \right)^{\frac{k-\ell}{K_1} -1 }  \\
	&=\frac{1}{1 - \big(1- \sigma(\gamma) \big) \eta}  \left(  \left( 1 - \big(1- \sigma(\gamma) \big) \eta \right)^{\frac{1}{K_1}}  \right)^{k-\ell},
	\end{aligned}
	\]
	where we defined $\prod_{t = 1}^0 x^t= 1$, for any sequence $\{x^t\}.$


\bibliographystyle{plain}
\bibliography{reference}

\end{document}